\documentclass[4pt, oneside]{article}

\usepackage[OT2, T1]{fontenc}
\usepackage[russian, english]{babel}
\usepackage{setspace}

\usepackage{yfonts}
\usepackage{amsmath,amssymb,amsthm}
\usepackage{units}
\usepackage{graphicx}
\usepackage[all]{xy}
\usepackage{bussproofs}
\usepackage{enumerate}
\usepackage[usenames,dvipsnames]{color}
\usepackage[colorlinks]{hyperref}
\hypersetup{
  colorlinks,
  citecolor=Violet,
  linkcolor=Red,
  urlcolor=Blue}
\usepackage{mathtools}

\newtheorem{theorem}{Theorem}[section]

\newtheorem{definition}[theorem]{Definition}
\newtheorem{lemma}[theorem]{Lemma}

\newtheorem{cor}[theorem]{Corollary}

\newtheorem{proposition}[theorem]{Proposition}

\newcommand{\adjective}{$\Delta_0$-}

\newcommand{\la}{\langle \, }
\newcommand{\ra}{\, \rangle}

\newcommand{\WO}{\mathbb{W}^{On}}
\newcommand{\al}{\alpha}
\newcommand{\be}{\beta}
\newcommand{\Gn}[1]{\ulcorner #1 \urcorner}
\newcommand{\RC}{\textbf{RC}}

\newcommand{\FLE}{\mathbb{F}_{\varepsilon_0}}
\newcommand{\FLEfin}{\texttt{Form}^0_{\mathcal{L}_{\omega}}}
\newcommand{\EA}{\text{EA}}
\newcommand{\TPr}{\textbf{\foreignlanguage{russian}{TS}C}}

\newcommand{\con}{\text{Con}}
\newcommand{\Cn}[2]{\underline{{\con}_{#1}({#2})}}
\newcommand{\gl}{\textbf{GL}}
\newcommand{\glp}{\textbf{GLP}}

\newcommand{\Th}{\sf Th}
\newcommand{\n}{\la n \ra}
\newcommand{\m}{\la m \ra}

\newcommand{\T}[1]{{\sf Th}_{#1}}
\newcommand{\suc}{\text{succ}}
\newcommand{\MNF}[2]{\la n_{#1}^{\al_{#1}} \ra \top \land \ldots \land \la n_{#2}^{\al_{#2}} \ra \top }

\begin{document}

\title{The logic of Turing progressions}

\author{Eduardo Hermo Reyes \footnote{\href{mailto:ehermo.reyes@ub.edu}{ehermo.reyes@ub.edu} } \\ Joost J. Joosten \footnote{\href{mailto:jjoosten@ub.edu}{jjoosten@ub.edu}} \\ {\small University of Barcelona}}
\maketitle

\begin{abstract}
Turing progressions arise by iteratedly adding consistency statements to a base theory. Different notions of consistency give rise to different Turing progressions. In this paper we present a logic that generates exactly all relations that hold between these different Turing progressions given a particular set of natural consistency notions. Thus, the presented logic is proven to arithmetically sound and complete for a natural interpretation, named the \emph{Formalized Turing progressions} (FTP) interpretation. 
\end{abstract}

\section{Introduction}

After G\"{o}del's incompleteness theorems \cite{Godel:1931}, we know that any consistent axiomatizable arithmetical theory, sufficiently rich, is incomplete. Moreover, by the second incompleteness theorem, we know that such a theory cannot prove the natural formalization of its own consistency. 

In other words, given a consistent arithmetical theory $T$ containing say, elementary arithmetic $\EA$, we have that $T \not \vdash \con(T)$ and therefore, extending $T$ by the consistency assertion $\con(T)$, we get a stronger theory. If we started with a sound theory $T$, the new system $T + \con(T)$ is affected once more by G\"{o}del's incompleteness theorems so it cannot prove its own consistency. Hence we can extend it by adding the corresponding consistency assertion obtaining a new, stronger, but again incomplete system. 

In 1936 Turing started working on his PhD thesis under the direction of Alonzo Church. The results of this work were published in 1939 under the title \emph{Systems of logic based on ordinals} \cite{Turing:1939:TuringProgressions}. Here, Turing introduced what are now known as Turing progressions which are hierarchies of theories that arise by iterating this process of adding consistency statements to a base theory, even to transfinite levels. Turing progressions are widely used in proof theory and can serve the purpose of gauging the proof-theoretic strength of mathematical theories that contain or interpret arithmetic. (See e.g.~\cite{Beklemishev:2004:ProvabilityAlgebrasAndOrdinals}.)

It is known that the standard propositional provability logic $\gl$ containing a modality $\Box$ to model the standard formalized provability predicate can be used to denote any finite Turing progression. The logic $\glp$ was introduced by Japaridze in \cite{Japaridze:1988:GLP} and contains a range of modalities $[n]$, one for each $n<\omega$, that can be interpreted as a sequence of ever increasing provability predicates. 

In \cite{Beklemishev:2004:ProvabilityAlgebrasAndOrdinals, Joosten:2014:TuringTaylor} it is shown that transfinite Turing progressions up to the ordinal $\varepsilon_0$ can be approximated using the polymodal provability logic $\glp$. In particular, \cite{Beklemishev:2004:ProvabilityAlgebrasAndOrdinals} showed how an ordinal analysis of Peano Arithmetic can be performed mainly within $\glp$.  

On the one hand, $\glp$ has many good algebraic properties and is very expressible in that, for example, its closed fragment constitutes for an alternative ordinal notation system. On the other hand, the link from $\glp$ to Turing progressions is only by approximating the progression so that applications of $\glp$ require many technical results. 

In this paper a first investigation is undertaken to see if there are logics which are still expressible and have good algebraic properties but which can be used to directly denote Turing progressions rather than just approximating them. In this sense, our logic is inspired by what is called Reflection Calculus $\RC$ (\cite{Beklemishev:2013:PositiveProvabilityLogic}, \cite{Beklemishev:2012:CalibratingProvabilityLogic}, \cite{Dashkov:2012:PositiveFragment}).
Like with $\RC$, we shall focus on strictly positive formulas which are formulas that do not contain disjunction nor implication. \\

The arrangement of the most important content of the current paper is as follows: Section \ref{Tprog} introduces some basic arithemtical definitions such as $n$-consistency and \emph{\adjective~presented theories}. In that section we shall also discuss the notion of \emph{smooth Turing progressions} and how they are presented within arithmetic. 

In Section \ref{signature} we introduce a propositional modal language using \emph{ordinal modalities}, which are modalities of the form $\la n^\al  \ra$ where $\al \in \varepsilon_0$ and $n \in \omega$ (named \emph{exponent} and \emph{base}, respectively). We will also establish a way to interpret our modal formulae as (possibly infinitely axiomatized) theories that contain a sufficient amount of arithmetic. This way, the intended reading of $\la n^\al \ra \varphi$ will be the  $n$-consistency $\al$ times iterated over the arithmetical interpretation of $\varphi$.

In Section \ref{TPr} we will introduce the system $\TPr$ (for \textbf{T}uring-\textbf{S}chmerl \textbf{C}alculus). The derivable objects of this system are \emph{sequents}, i.e., expressions of the form $\varphi \vdash \psi$ whose intended interpretation is to express the entailment between the theories denoted by the modal formulas $\varphi$ and $\psi$. First rank inhabitants of $\TPr$ are the so-called \emph{monomial normal forms} ({\sf MNF}'s for short) and it is shown that each formula is equivalent to such a {\sf MNF}.

In Section \ref{arithint}, we will see how to fix the interpretation for these sequents within $\EA^+$ and we prove the soundness of $\TPr$ with respect to this interpretation. In Section \ref{boundCompl}, we introduce some conservativity results and a characterization for the derivability between {\sf MNF}'s. With these results we prove the completeness of the system.

\section{Turing progressions} \label{Tprog}

We will focus on theories that contain a $\Pi^0_1$ formulation of $\EA^+$. By doing so, the logical complexity of our theories can be low yet various arguments requiring cut-elimination can be formalized. Here, $\EA^+$ is Robinson's arithmetic $\text{Q}$ together with induction for bounded formulas. Note that $\EA^+$ in our formulation has a function symbol for the super-exponential function $x\mapsto 2^x_x$ defined by $2^x_0 := x$ and $2^x_{y+1}:=2^{2^x_y}$.

\subsection{Arithmetical preliminaries}

Since we will be dealing with formalized syntax it is very important to really understand how we represent theories. We will actually identify a theory with the formula that defines the set of G\"odel numbers of its axioms. However, we do not want the complexity of this defining formula to be too high. Therefore we only consider so-called \emph{\adjective presented theories}:

\begin{definition}[\adjective presented theories]
A  theory $U$ is \adjective presented iff there is a bounded arithmetical formula $\sigma (x)$ that defines the sets of G\"{o}del numbers coding the axioms of $U$ in the standard model of arithmetic. By $\epsilon (x)$ we denote the canonical presentation of $\EA^+$.
\end{definition}

In the remainder of this paper we will consider \adjective~presented theories although many results presented here can be proven in a more general setting. Note that since supexp is in our language, our notion of \adjective~presentable does not coincide with the more standard one from the literature where one considers theories whose set of axioms are definable via an \emph{elementary} formula (a $\Delta_0$ formula that does not contain supexp and where only a function symbol for exponentiation is allowed). Via Craig's trick, we think that this difference is not too important.

As usual, by $\Box_T (x)$ we denote $\exists y \, \text{Prf}_T (y, x)$; an arithmetical $\Sigma_1$-formula expressing that there is $y$ such that $y$ codes a proof in $T$ of a formula with code $x$. By $\overline{n}$ we denote the \emph{numeral} $S(S( \ldots S(0) \ldots ))$  ($n$ times). We will use just $n$ when it cannot be mislead with a variable. Often, we write $\Box_T (\varphi (\dot{x}) )$ instead of $\Box_T (\Gn{\varphi (\dot{x})})$, where $\Gn{\varphi(\dot{x_1}, \ldots , \dot{x_j})}$ denotes the map sending $n_1 , \ldots , n_k$ to the G\"{o}del number $\Gn{\varphi(\overline{n_1}, \ldots , \overline{n_j})}$. 

\begin{definition}[$n$-consistency]
A theory $U$ is called $n$-consistent if $U$ together with the set of all true $\Pi_{n}$-sentences is consistent. More formally:
\[
{\sf Cons}_n (U) \ := \ \forall \pi \ (\sf {Tr}_{n}(\pi) \to \neg \Box_U \neg \pi).
\]
Here, $\sf {Tr}_{n}$ is the standard $\Pi_{n}$-truth definition for $\Pi_{n}$-formulas.
\end{definition}

The $n$-consistency of the theory $U$ can equivalently (in $\EA^+$) be expressed (see e.g.~\cite{Beklemishev:2005:Survey}) by the arithmetical formula:
\begin{equation}
\con_n (U):= \hspace{0.25cm} \forall x \in \Pi_{n+1}\ (\Box_U (x) \rightarrow \sf {Tr}_{n+1}(x))
\end{equation}
where $x \in \Pi_{n+1}$ expresses that $x$ is the G\"{o}del number of a $\Pi_{n+1}$ sentence. Notice that the arithmetical complexity of $\con_n(U)$ is $\Pi_{n+1}$.

Turing progressions iterate consistency along a well-order. Thus to speak about Turing progression within arithmetical theories we first need to settle on how to represent linear orders and well-orders within such theories.

\begin{definition}[Elementary linear ordering, elementary well-order]
A pair $(D, \prec)$ is an elementary linear ordering iff both $D$ and $\prec$ are elementary defined, $D \subseteq \mathbb{N}$, $\prec \ \subseteq D^2$ and $\EA^+$ proves that $\prec$ linearly orders $D$. An elementary linear ordering is a well-order if it is well-founded in the standard model.
\end{definition}

For the remainder of this paper we shall furthermore assume that our elementary orders are nice in that the basic properties of and operations on the elements of the ordering are available in $\EA^+$. As such we can view the elements as ordinals, we can distinguish successor ordinals from limit ordinals and perform the basic operations like addition, multiplication and exponentiation.

\subsection{Smooth Turing progressions}

\emph{Turing progressions} are hierarchies of theories such that, given an initial theory $T$, we can construct a transfinite sequence of extensions of $T$ by iteratedly adding $n$-consistency statements. These progressions can be defined according to the following conditions below. Let us assume that we have fixed some recursive limit ordinal $\Lambda$ together with a natural ordinal notation system for $\Lambda$ that is elementary presented. Turing progressions are essentially defined by the following three clauses.

\begin{enumerate}[T1.]
\item $(T)^0_n := T$ where $T$ is an initial or base theory;

\item $(T)^{\al+1}_n := (T)^\al_n \cup \{\,  \con_n \big( \, (T)^\al_n \, \big) \, \}$;

\item $(T)^\lambda_n := \bigcup_{\be < \lambda} \, (T)^\be_n$,  \ for $\lambda$ a limit ordinal not exceeding $\Lambda$.

\end{enumerate}

There are various ways one can define/present a series of theories that satisfy these three clauses. Here we will consider \emph{smooth Turing progressions}, studied by Beklemishev in among others \cite{Beklemishev:2005:Survey} and \cite{Beklemishev:1991:ProvabilityLogicsForNaturalTuringProgressions}. We give a slightly different presentation.

Suppose we are given some elementary well-ordering $(D, \prec )$ and an initial theory $T$. The conditions T1-T3 can be reformulated by the unique following clause:

$$(T)^\al_n := T \cup \{ \, \con_n \big( \, (T)^\be_n \, \big) : \ \prec(\be , \al) ,  \ \be \in D \, \} \ \ \ \forall \al \in D, \  n < \omega$$.


Note that there is a seeming circularity in this definition. Of course, by means of the fixpoint theorem we can find a formalization of the definition within arithmetic.

We say that $\tau_n^{\sigma(z)} (\al, x)$ enumerates the $\al$-th theory of a progression based on iteration of $n$-consistency along $(D, \prec)$ with base $\sigma(z)$ if:
\[
\begin{array}{ll}
EA^+ \vdash \tau_n^{\sigma (z)} (\al, x) \ \leftrightarrow \ \Big( & \big(\, \epsilon (x) \vee \sigma (x) \, \big) \  \vee \\
\ &\exists \be \ \big(\prec (\be , \al) \ \wedge \ x = \Gn {\con_n (\tau_n^{\sigma (z)} (\dot{\beta}, y))} \ \big)\ \Big).
\end{array}
\]

Recall that by our reading conventions we have that $\con_n (\tau_n^{\sigma (z)} (\dot{\beta}, y))$ denotes the $n$-consistency of the theory axiomatized by $\tau_n^{\sigma (z)} (\dot{\beta}, y)$.
Since the equivalence does not directly refer to $\con_n (\tau_n^{\sigma (z)} (\dot{\beta}, y))$ but rather to its G\"odel number, the existence of such $\tau_n^{\sigma (z)} (\al, x)$ is guaranteed by the fixed point theorem. Furthermore, in \cite{Beklemishev:1995:LocalRefVSCon} and \cite{Beklemishev:2003:AnalysisIteratedReflection} it is shown that given an elementary presented theory $T$ and an elementary well-ordering $(D, \prec)$, there is a unique progression $(T)_n^\al$ for $n < \omega$ and $\al \in D$, modulo provable equivalence in $\EA^+$.

\section{A modal language for Turing progressions} \label{signature}

As we have pointed out, we shall work with a strictly positive propositional modal signature consisting of one constant symbol $\top$, one logical connective $\wedge$ and a set of modal connectives $\mathcal{M} := \{ \la n^\al \ra$ : $n < \omega$ and $\al < \varepsilon_0 \}$, named \emph{ordinal modalities}. The set of formulas in this language is defined as follows:

\begin{definition}
By $\FLE$ we denote the smallest set such that:

\begin{enumerate}[i)]
\itemsep0em
\item $\top \in \FLE$;
\item If $\varphi, \, \psi \in \FLE$ then $(\varphi \wedge \psi) \in  \FLE$;
\item If $\varphi \in \FLE, \ n<\omega$ and $\alpha < \varepsilon_0$ then $\la n^\alpha \ra \, \varphi \in \FLE$.
\end{enumerate}
\end{definition}

At the end of this section we shall see how these formulas can be associated to arithmetical theories in a natural way. To have some control over the nature of these theories, for any formula $\psi$ in the signature we define the following two functions returning the set of base and exponent elements respectively of any modality occurring in $\psi$. That is:s

\begin{definition}
Given a formula $\psi$, by ${\sf N\text{-}mod}(\psi)$ we denote the set of natural numbers corresponding to the base element in each modality that occurs in $\psi$, i.e.,
\begin{enumerate}[i)]
\item ${\sf N\text{-}mod}(\top) = \emptyset$;
\item ${\sf N\text{-}mod}(\varphi \, \wedge \, \psi) = {\sf N\text{-}mod}(\varphi) \, \cup \, {\sf N\text{-}mod}(\psi)$;
\item ${\sf N\text{-}mod}(\la n^\al \ra \, \psi) = \{ n \} \cup {\sf N\text{-}mod}(\psi)$.	
\end{enumerate}

Analogously, by ${\sf O\text{-}mod}(\psi)$ we denote the set of ordinals below $\varepsilon_0$ corresponding to the exponent element of each modality in $\psi$, i.e.,
\begin{enumerate}[i)]
\item ${\sf O\text{-}mod}(\top) = \emptyset$;
\item ${\sf O\text{-}mod}(\varphi \, \wedge \, \psi) = {\sf O\text{-}mod}(\varphi)\, \cup \, {\sf O\text{-}mod}(\psi)$;
\item ${\sf O\text{-}mod}(\la n^\al \ra \, \psi) = \{ \al \} \cup {\sf O\text{-}mod}(\psi).$
\end{enumerate}
\end{definition}
 
\subsection{Ordinal worms and restricted sets of formulas}

Within $\FLE$ we can find some special formulas, named \emph{Ordinal Worms} -OWs-, which are defined as follows:

\begin{definition}[Ordinal Worms, $\WO$]
The set of OWs denoted by $\WO$ is inductively defined as:
\begin{enumerate}[i)]
\itemsep0em
\item $\top \in \WO$;
\item If $B \in  \WO$, then $\la n^\alpha \ra \, B \in \WO$ for any $n < \omega$ and $\al < \varepsilon_0$ .
\end{enumerate}
\end{definition}

In order to control the complexity of the theories corresponding to our formulas, for any $n<\omega$ we define the set of formulas $\mathbb{F}_{<n}$ as follows:

\begin{definition} $\mathbb{F}_{<n}$ is the smallest set such that:
\begin{enumerate}[i)]
\item $\top \in \mathbb{F}_{<n}$;
\item if $\varphi, \, \psi \in \mathbb{F}_{<n}$ then $(\varphi \, \wedge \, \psi) \in \mathbb{F}_{<n}$;
\item if $\varphi \in \mathbb{F}_{<n}$ then $\la m^\be \ra \, \varphi \in \mathbb{F}_{<n}$ for and $m < n$ and $\be < \varepsilon_0$.
\end{enumerate}
\end{definition}

This last set of formulas will have some special relevance when trying to capture some conservativity results.\\

From now on, we let Greek letters $\al , \be ,\gamma , \ldots$ denote ordinals below $\varepsilon_0$, Greek lower case in the middle of the alphabet like $\varphi, \psi, \chi$ denote formulas, and capital latin letters $A, B, C, \ldots$ denote OWs. Also, given $A \in \WO$ of the form $\la n_0^{\al_0 }\ra \ldots \la n_k^{\al_k} \ra \top$, and $\varphi \in \FLE$ by $A\varphi$ we denote the formula $\psi$ obtained by the concatenation $\la n_0^{\al_0} \ra \ldots \la n_k^{\al_k} \ra \, \varphi$.

\subsection{Monomial and increasing normal forms}

In this subsection we consider two different kind of special formulas named \emph{monomial normal forms} and \emph{increasing normal forms}. Formulas in monomial normal form are used in the axiomatization of the calculus $\TPr$ as introduced in the next section and play an important role in this paper. 

Later we will show that for every formula $\psi$ in monomial normal form there is a unique equivalent (modulo $\TPr$) OW  $A$ in increasing normal form and viceversa. Moreover, we shall show that for any formula $\varphi$ in the above signature, there is a unique equivalent $\psi$ in monomial normal form. These uniqueness claims shall be proved after proving the arithmetical soundness.

\begin{definition}
The set {\sf M}  of \emph{monomials} is defined as the set of all OWs $A$ of length 1 i.e. ${\sf M} := \{\, \la n^\al \ra \top  \in \mathbb{W}^{On} : \text{ for some } n < \omega \text{ and } \al < \varepsilon_0 \, \}$.
\end{definition}

Monomial normal forms are conjuntions of monomials with an additional condition on the occuring exponents. In order to formulate this condition we first need to define the \emph{hyper-exponential} as studied in \cite{FernandezJoosten:2012:Hyperations}.

\begin{definition}[hyper-exponentiation]
The \emph{hyper-exponential} functions $e^n: \text{On} \rightarrow \text{On}$ where On denotes the class of ordinals, $e^0 $ is the identity function, $e^1: \al \mapsto -1+\omega^\al$ and $e^{n+m}= e^n \circ e^m$.

\end{definition} 
We will use $e$ to denote $e^1$. Note that for $\alpha$ not equal to zero we have that $e(\alpha)$ coincides with the regular ordinal exponentiation with base $\omega$; that is, $\alpha \mapsto \omega^\alpha$. However, it turns out that hyper-exponentials have the nicer algebraic properties in the context of provability logics.  With hyper-exponentiations at hand we can now define the notion that can be considered the corner-stone of our calculus.

\begin{definition}\label{definition:MNFs}
The set of formulas in \emph{monomial normal form}, {\sf MNF}, is inductively defined as follows:
\begin{enumerate}[i)]
\item $\top \in {\sf MNF}$;
\item If $A \in {\sf M}$ then $A \in {\sf MNF}$;
\item \label{item:TechnicalMNFclause}
\begin{tabbing}
	  If \hspace{0.2cm} \= $(a) \ $ \= $\la n_0^{\al_0} \ra \top \wedge \ldots \wedge \la n_k^{\al_k} \ra \top \in {\sf MNF}$; \\
 	  \	\\	     
	     \> $(b) \ $ \> \ $n < n_0$; \\
	  \ \\
	     \> $(c) \ $ \> \ $\al$ of the form $e^{n_0 - n}(\al_0) \cdot (2 + \be)$ for some $\be < \varepsilon_0$,
      \end{tabbing}
      \vspace{0.25cm}
then $\la n^\al \ra \top \wedge \la n_0^{\al_0} \ra \top \wedge \ldots \wedge \la n_k^{\al_k} \rangle \top \in {\sf MNF}.$
\end{enumerate}
\end{definition} 

Thus, monomial normal forms are conjunctions of monomials where the base of the modalities run in increasing order from left to right. The technical Item \ref{item:TechnicalMNFclause} is there to ensure that only monomials are included in the {\sf MNF} if they add new information to the {\sf MNF} that was not already implicit by one of the other monomials.

As we shall later see, there is another natural class of formulas that provides us with alternative normal forms. We define this set already here. 

\begin{definition}
The set of formulas in \emph{increasing normal form}, {\sf INF}, is inductively defined as follows:
\begin{enumerate}[i)]
	  \item $\top \in {\sf INF}$;
	  \item if $\la m^\be \ra A \in {\sf INF}$, $n < m$ and $0 < \al$, then $\la n^\al \ra \la m^\be \ra A \in {\sf INF}$.
\end{enumerate} 
\end{definition}

\subsection{Arithmetical interpretation of modal formulas}

Let us introduce now the arithmetical interpretation of our modal formulae in terms of the $\tau$-formulae previously presented. Notice that the conjunction of modal formulas is intended to mean the union of theories. Hence, we map the interpretation of modal conjunctions to the disjunction of the respective interpretations.

\begin{definition} \label{def}
An arithmetical interpretation for $\FLE$ is a map \\$* : \FLE \longrightarrow \texttt{Form}_{\mathcal{L}_\mathbb{N}}$ inductively defined as follows:
\begin{enumerate}
\itemsep0em
\item $(\top )^*(x)= \epsilon (x)$;
\item $(\varphi \, \wedge \, \psi)^*(x)= (\varphi)^*(x) \vee (\psi)^*(x)$
\item $(\la n^\al \ra \, \varphi)^*(x)= \tau_n^{\varphi^* (y)}( \al, x)$.
\end{enumerate}
\end{definition}

Note that since $\FLE$ has no propositional variables, we can identify a modal formula with its arithmetical interpretation unambiguously and so we will do in the remainder of the paper.

\begin{definition}
Given a numeration $\sigma$, by ${\sf Th}_\sigma$ we denote the theory of $\sigma$. We say that $\chi \in {\sf Th}_\sigma$ iff $\mathbb{N} \models \sigma (\Gn{\chi})$.
\end{definition}

For the sake of clarity, and since we are working in the close fragment, we will use the following notation: given $\varphi \in \FLE$ by ${\sf Th}_\varphi$ we denote ${\sf Th}_\sigma$ where $\varphi^*(x)=\sigma(x)$, following Definition \ref{def}. If $\varphi^*(x)= \epsilon (x)$ we use just $\EA^+$ instead of ${\sf Th}_\epsilon$. Also, if $\varphi := \la n^\al \ra \psi$ we write $({\sf Th}_\psi)_n^\al$ for the theory of $(\la n^\al \ra \psi)^*$.

\section{Turing-Schmerl calculus} \label{TPr}

In the previous Section we have defined an arithmetical interpretation that allows us to denote Turing progressions by modal formulas in $\FLE$. In this section we introduce the logic $\TPr$ in this modal language whose main goal is to express valid relations that hold between the corresponding Turing progressions.\\

We will use the following notation too: by $\varphi \equiv \psi$ we will denote that both $\varphi \vdash \psi$ and $\psi \vdash \varphi$ are derivable. Analogously, by $\varphi \equiv_n \psi$ we denote that for any formula $\chi \in \mathbb{F}_{<n +1}$, $\varphi \vdash \chi$ iff $\psi \vdash \chi$, i.e.,  $\varphi$ and $\psi$ share the same $\mathbb{F}_{<n +1}$ consequences. Also, by convention we take that for any $n$, $\la n^0 \ra \varphi$ is just $\varphi$.

\begin{definition} 
$\TPr$ is given by the following set of axioms and rules:\\

Axioms:
\begin{enumerate}

\item 
$\varphi \vdash \varphi, \ \ \ \varphi \vdash \top$;	

\item 
$\varphi \wedge \psi \vdash \varphi,  \ \ \ \varphi \wedge \psi \vdash \psi$;	\label{conjel}

\item \label{mon}
$\la n^\al \ra \varphi \vdash \la n^\be \ra \varphi$, \ \ \ for $\be \leq \al$;	
		
\item  \label{coadditive}
Co-additivity axioms: $\la n^{\al + \be} \ra \, \varphi \, \equiv\, \la n^\be \ra \la n^\al \ra \, \varphi$;	 

\item \label{omega}
$\la (m + n)^\al \ra \, \varphi \, \vdash \, \la m^{e^n (\al)} \ra \, \varphi$;	

\item \label{MS1}
Schmerl axioms:
\[
\begin{array}{ll}
\la n^\al \ra \, \big(\la n_0^{\al_0} \ra \top \wedge \ldots \wedge \la n_k^{\al_k} \rangle \top\big) \ \equiv & \la n^{e^{n_0 - n} (\al_0) \cdot (1 + \al)}  \ra \, \top \ \  \land \\
 & \la n_0^{\al_0} \ra \top \land \ldots \land \la n_k^{\al_k} \ra \top\\
\end{array}
\]
\ \\
 \ \ \ for $n < n_0$ and $\MNF{0}{k} \in {\sf MNF}$. 

\end{enumerate}
\

Rules:
\begin{enumerate}
\item \label{r:1}
If $\varphi \vdash {\psi}$ and ${\phi}\vdash{\chi}$, then $\varphi\vdash {\psi}\wedge{\chi}$;	

\item \label{r:2}
If $\varphi\vdash {\psi}$ and ${\psi}\vdash{\chi}$ then $\varphi\vdash{\chi}$;	

\item \label{r:3}
If $\varphi\vdash {\psi}$, then $\la n^\al \ra \varphi\vdash \la n^\al \ra {\psi}$ ;	

\item \label{r:4}
If $\varphi \vdash \psi$ then $\la n^\al \ra \varphi \ \land \ \la m^{\be + 1} \ra \psi \, \vdash \, \la n^\al \ra \, \Big( \, \varphi \ \land \ \la m^{\be + 1} \ra \psi \, \Big) \ \ \ \text{ for } m < n$. 
\end{enumerate}
\end{definition}

It is worth mentioning the special character of Axioms \ref{omega} and \ref{MS1} since both axioms are modal formulations of principles related to Schmerl's fine structure theorem, also known as \emph{Schmerl's formulas} (see \cite{Schmerl:1978:FineStructure} and \cite{Beklemishev:2003:AnalysisIteratedReflection}) as formulated in Proposition \ref{Schmerl} of this paper.

The following Proposition expresses some useful properties that can be easily proved.

\begin{proposition} For any $A \in \WO$ and $\varphi, \psi \in \FLE$ : \label{prop}
\begin{enumerate}
\item 
If $\varphi \vdash \psi$, then $A\varphi \vdash A\psi$; \label{prop1}

\item 
$A \varphi \vdash \varphi \land A$; \label{prop3}

\item 
$\la (m + n)^\al \ra \, \varphi \, \vdash \, \la m^\al \ra \, \varphi$; \label{prop4}

\item $\la n^\al \ra (\, \varphi \, \wedge \, \psi \, ) \ \wedge \ \la m^{\be + 1} \ra \varphi \equiv \la n^\al \ra ( \la m^{\be + 1} \ra \varphi \ \land \ \psi)$ for $m < n$; \label{prop5}

\item $\la n^\al \ra \varphi \, \land \, \psi \equiv \la n^\al \ra ( \varphi \, \land \, \psi )$ for $\psi \in \mathbb{F}_{< n}$ and ${\sf O}\text{-}{\sf mod}(\psi) \subset \suc$.  \label{prop6}

\end{enumerate}
\end{proposition}

\begin{proof}
The only non-trivial item is Item \ref{prop6} which follows from a straight-forward induction on the length of $\psi$ using Rule \ref{r:4}.
\end{proof}

\section{Normal forms}


In this section we shall prove that any formula of $\TPr$ is equivalent to a formula in monomial normal for --${\sf MNF}$-- and also equivalent to a formula in increasing normal form, ${\sf INF}$.

\subsection{Monomial normal forms versus increasing normal forms}

In this subsection we prove that monomial normal forms and increasing normal forms can easily be transformed into each other. Let us first observe that for each $\sf MNF$ there is an equivalent $\sf INF$ that holds some similarities with the original $\sf MNF$.

\begin{theorem}
For every $\psi \in {\sf MNF} \mbox{ with } \ \psi := \MNF{0}{k}$, there is an ordinal worm $A \in {\sf INF}$ such that:
\begin{enumerate}
\item 
$\psi \equiv A$;

\item $A := \la n_0^{\be_0} \ra \ldots \la n_k^{\be_k} \ra \top$ where:
	\begin{enumerate}
	\item $\be_k = \al_k$ and,
	\item for all $i, \ 0 \leq i < k$ we have
	\[
	\be_i = -1 + \frac{\al_i}{e^{n_{i+1} - n_i}(\al_{i+1})} .
	\]
	\end{enumerate}
\end{enumerate}	
 \label{MNFINC}
\end{theorem}

\begin{proof}
By induction on $k$. The base case is trivial and the inductive case follows from the I.H. and Axiom \ref{MS1}. Note that by an easy induction we see that the division is well-defined in virtue of Item \ref{item:TechnicalMNFclause} of Definition \ref{definition:MNFs}.
\end{proof}

In particular, this theorem tells us that if we start out with a $\sf MNF$ $\psi$, there is an equivalent $\sf INF$ ordinal worm $A$ where the modality bases are the same. The next theorem tells us that this feature of conserving the modality bases also holds if we go from an ordinal worm $A$ in $\sf INF$ to an equivalent formula $\psi$ in $\sf MNF$.

\begin{theorem}
Let $A := \la n_0^{\be_0} \ra \ldots \la n_k^{\be_k} \ra \top \in {\sf INF}$. There is $\psi \in {\sf MNF}$ such that:
\begin{enumerate}
\item 
$A \equiv \psi$;

\item $\psi := \MNF{0}{k}$ where
\begin{enumerate}
\item
$\al_k = \be_k$ and,

\item
for all $i, \ 0 \leq i < k$, $\al_i = e^{n_{i+1} - n_i} (\be_{i+1}) \cdot (1 + \be_i)$.
\end{enumerate}
\end{enumerate}
\label{INCMNF}
\end{theorem}
\begin{proof}
By induction on $k$, with the help of Axiom \ref{MS1} for the inductive case.
\end{proof} 

The proofs of these theorems actually provide operations $\mathcal I$ and $\mathcal M$ to go from a formula $\psi$ in $\sf MNF$ to a corresponding $\mathcal I (\psi)$ in $\sf INF$ and from an ordinal worm $A$ in $\sf INF$ to a corresponding formula $\mathcal M (A)$ in $\sf MNF$. It is easy to see that both $\mathcal M \circ \mathcal I$ and $\mathcal I \circ \mathcal M$ are the identity.

\begin{cor} \label{peelHead}
Given $\psi \in {\sf MNF}$ of the form $\MNF{0}{k}$ and $A := \la n_0^{\be_0} \ra \ldots \la n_k^{\al_k} \ra \top \in {\sf INF}$, if $\mathcal{I}(\psi) = A$ then $\mathcal{I}(\MNF{1}{k}) = \la n_1^{\be_1} \ra \ldots \la n_k^{\al_k} \ra \top$. 
\end{cor}

\subsection{Monomial normal forms are closed under conjunctions}

In this subsection we will prove that $\sf MNF$s are closed under taking conjunctions. A first step in doing so is established by the following lemma.

\begin{lemma} \label{conNF}
Let $\psi := \MNF{0}{k} \in {\sf MNF}$ and $n < n_0$. There is $\psi' \in {\sf MNF}$ such that:
\[
\la n^\beta \ra \top \land \psi \equiv \psi'. 
\]
\end{lemma}

\begin{proof}
If $\beta \leq e^{n_0 - n} (\al_0)$, clearly we can take $\psi' = \psi$. Otherwise, let $\beta_{\sf CNF}$ be the Cantor normal form of $\beta$ with base $\omega$ and we write  $\beta_{\sf CNF} := \beta_0 + \ldots + \beta_i$ with each $\beta_j$ additively indecomposable and $\beta_0 \geq \ldots \geq \beta_i$. 

If for all $j$ with $0 \leq j \leq i$ we have $\al_j \geq e^{n_0 - n} (\al_0)$, then it easy to see that $\la n^\al \ra \top \land \psi \in {\sf MNF}$. 
If on the contrary, there is $j$  with $0 < j \leq i$ such that $\beta_j < e^{n_0 - n} (\al_0)$ and $\beta_{j-1} \geq e^{n_0 - n} (\al_0)$, then we reason as follows:
\begin{tabbing}
$\la n^\beta \ra \top \land \psi $ \= $\equiv \la n^{\beta} \ra \top \land \la n_0^{ \al_0} \ra \top \land \psi$;\\ [0.20cm]
\ \> 
$\equiv \la n^{\beta_0 + \ldots + \beta_{i}} \ra \top \land \la n_0^{ \al_0} \ra \top \land \psi$;\\ [0.20cm]
\ \> 
$\vdash \la n^{\beta_0 + \ldots + \beta_{j-1} + 1 } \ra \top \land \la n_0^{ \al_0} \ra \top \land \psi$;\\[0.20cm]
\ \> 
$\vdash \la n_0^{ \al_0} \ra \la n^{\beta_0 + \ldots + \beta_{j-1} + 1 } \ra \top \land \psi$;\\[0.20cm]

\ \> 
$\vdash \la n^{ e^{n_0 -n}(\al_0)} \ra \la n^{\beta_0 + \ldots + \beta_{j-1} + 1 } \ra \top \land \psi$;\\[0.20cm]

\ \> 
$\vdash \la n^{\beta_0 + \ldots + \beta_{j-1} + e^{n_0 - n} (\al_0)} \ra \top \land \psi$.
\end{tabbing} 

However, by monotonicity we see that $\la n^{\beta_0 + \ldots + \beta_{j-1} + e^{n_0 - n} (\al_0)} \ra \top \vdash \la n^\beta \ra \top$ so that actually $\la n^{\beta_0 + \ldots + \beta_{j-1} + e^{n_0 - n} (\al_0)} \ra \top \land \psi \equiv \la n^\beta \ra \top \land \psi$.
Clearly we have $\la n^{\beta_0 + \ldots + \beta_{j-1} + e^{n_0 - n} (\al_0)} \ra \top \land \psi \in {\sf MNF}$ and we are done.
\end{proof}

We get the following simple corollaries.

\begin{cor}\label{theorem:conjunctionOfMonomialsHasMNF}
Any conjunction of monomials is equivalent to some ${\sf MNF}$.
\end{cor}

\begin{proof}
By the length of such a conjunction of monomials written with the modality bases from small to large, applying Lemma \ref{conNF} at the inductive step. 
\end{proof}

We can recast it as the first inductive step that each formula is equivalent to a $\sf MNF$.

\begin{cor}\label{theorem:MNFsClosedUnderConjunctions}
The monomials normal forms --$\sf MNF$s-- are closed under conjunctions.
\end{cor}

In the light of this corollary, it only remains to show that $\sf MNF$s are closed under putting an ordinal-diamond up front of them.

\subsection{Monomial normal forms are closed under ordinal diamonds}

We shall now see that the collection of monomial normal forms is, modulo provable equivalence, closed under putting ordinal diamonds up front of them. 
Let us recall Schmerl's axiom scheme: 
\[
\begin{array}{ll}
\la n^\al \ra \, \big(\la n_0^{\al_0} \ra \top \wedge & \ldots \ \wedge \la n_k^{\al_k} \rangle \top\big) \ \ \ \ \ \ \ \equiv\\
&  \la n^{e^{n_0 - n} (\al_0) \cdot (1 + \al)}  \ra \, \top \ \  \land 
\la n_0^{\al_0} \ra \top \land \ldots \land \la n_k^{\al_k} \ra \top.\\
\end{array}
\]
We observe that this axiom tells us that putting a relatively small ($n< n_0$)  non-trivial ($\alpha \geq 1$) ordinal modality up front an {\sf MNF} $\MNF{0}{k}$ yields a new {\sf MNF}. To see that putting \emph{any} ordinal modality up front an {\sf MNF} is equivalent to some other {\sf MNF}, we need to study two cases: $n=n_0$ and $n>n_0$. Let us start with the latter.

\begin{proposition}\label{theorem:PseudoSchmerlOne}
Let $n, n_0$ be natural numbers with $n > n_0$. Further, let $\la n_0^{\al_0} \ra \top \, \land \, \MNF{1}{k}$ with $k\geq 0$ be a conjunction of monomials not necessarily in ${\sf MNF}$.
The following principle (we use $\text{PS}$ for \emph{Pseudo Schmerl}) is derivable in $\TPr$:\label{ObsMS2}
\[
\begin{array}{ll}
\la n^\al \ra \Big(\, \la n_0^{\al_0} \ra \top \  \wedge & \MNF{1}{k}\, \Big) \ \ \equiv \ \ \ \ \ \ \ \ \ \ \  \  \ \ \ \ \ \ \ \ \ \ \ \ \ \ \text{(PS1)}\\
& \la n_0^{\al_0 + e^{n-n_0} (\al)} \rangle \top \ \land \ \la n^\al \ra \Big( \, \la n_1^{\al_1} \ra \top \land \ldots \land \la n_k^{\al_k} \ra \top \, \Big). \\
\label{PS1}
\end{array}
\]
Note that for the case when $k=0$, the formula $\MNF{1}{k}$ is to be read as the empty conjunction which is $\top$ so that the principle becomes
\[
\la n^\al \ra  \la n_0^{\al_0} \ra \top  \ \ \equiv \ \  \la n_0^{\al_0 + e^{n-n_0} (\al)} \rangle \top \ \land \ \la n^\al \ra \top. 
\]
\end{proposition}
\begin{proof}
For the right-to-left sequent, we apply the following reasoning:
\[
\begin{array}{ll}
\la n_0^{\al_0 + e^{n-n_0} (\al)} \ra \top \land \la n^\al \ra \Big(& \hspace{-0.3 cm}\la n_1^{\al_1} \ra \top \land \ldots \land \la n_k^{\al_k} \ra \top \Big) \\
 &\vdash \ \ \  \la n_0^{\al_0 + 1} \ra \top \wedge \la n^\al \ra \Big(\MNF{1}{k} \Big);\\
 &\vdash \ \ \   \la n^\al \ra \Big( \la n_0^{\al_0 + 1} \ra \top \, \land \, \MNF{1}{k}\Big);\\
 &\vdash \ \ \   \la n^\al \ra \Big( \la n_0^{\al_0} \ra \top \, \land \, \MNF{1}{k}\Big). \\
\end{array}
\]

For the left-to-right sequent we can reason as follows:
\[
\begin{array}{ll}
\la n^\al \ra \Big(\la n_0^{\al_0} \ra \top \, \wedge & \ldots \ \wedge \la n_k^{\al_k} \ra \top \Big)  \\
 &\vdash \ \ \ \ \la n^\al \ra \la n_0^{\al_0} \ra \top \land \la n^\al \ra \Big(\MNF{1}{k} \Big);\\
 &\vdash \ \ \ \  \la n_0^{e^{n-n_0} (\al)} \ra \la n_0^{\al_0 } \rangle \top \land \la n^\al \ra \Big( \MNF{1}{k} \Big);\\
 &\vdash \ \ \ \  \la n_0^{\al_0 + e^{n-n_0} (\al)} \rangle \top \land \la n^\al \ra \Big(\MNF{1}{k}\Big). \\
\end{array}
\]

\end{proof}

\begin{proposition}\label{theorem:smallModalityOverEqualMNF}
Let $\MNF{0}{k} \in {\sf MNF}$ for $k\geq 1$. 
The following principle is derivable in $\TPr$:
\[
\begin{array}{ll}
\la n_0^\al \ra \, \Big(\, \la n_0^{\al_0} \ra \top & \wedge \ \ldots \ \wedge \la n_k^{\al_k} \ra \top \Big)\\
\ & \equiv \ \ \ \la n_0^{\al_0 + e^{n_1-n_0} (\al_1) \cdot \al} \rangle \, \top \ \land \ \la n_0^\al \ra \, \Big(\, \la n_1^{\al_1} \ra \top \land \ldots \land \la n_k^{\al_k} \ra \top \, \Big).\\
\end{array}
\]

\end{proposition}

\begin{proof}
For the left-to-right sequent, first notice that 
\[
\la n_0^\al \ra (\la n_0^{\al_0} \ra \top \wedge \ldots \wedge \la n_k^{\al_k} \ra \top) \vdash \la n_0^\al \ra ( \la n_1^{\al_1} \ra \top \land \ldots \land \la n_k^{\al_k} \ra \top ).
\]
To obtain that furthermore
\[
\la n_0^\al \ra (\la n_0^{\al_0} \ra \top \wedge \ldots \wedge \la n_k^{\al_k} \ra \top) \vdash \la n_0^{\al_0 + e^{n_1-n_0} (\al_1) \cdot \al} \rangle \top,
\]
note that by means of Theorem \ref{MNFINC} together with Axiom \ref{coadditive} we have that
\[
\la n_0^\al \ra \, \Big(\la n_0^{\al_0} \ra \top \wedge \ldots \wedge \la n_k^{\al_k} \ra \top\, \Big) \ \vdash \ \la n_0^{\delta_0 + \al} \ra \la n_1^{\delta_1} \ra \ldots \la n_k^{\al_k} \ra \top
\]
where $\delta_i = -1 + \frac{\al_i}{e^{n_{i+1} - n_i}(\al_{i+1})}$ for $0 \leq i < k$. Thus, by Theorem \ref{INCMNF} we get via an easy induction that:

\[
\la n_0^{\delta_0 + \al} \ra \la n_1^{\delta_1} \ra \ldots \la n_k^{\al_k} \ra \top \ \equiv \ \la n_0^{\delta_0 + \al} \ra \, \Big( \MNF{1}{k} \Big).
\]

Now, we apply Axiom \ref{MS1} once more observing that  
\[
e^{n_1 - n_0} (\al_1) \cdot (1 + \delta_0 + \al) = e^{n_1 - n_0} (\al_1) \cdot (1 + (-1 + \frac{\al_0}{e^{n_{1} - n_0}(\al_{1})}) + \al)\]

that is,
\[
e^{n_1 - n_0} (\al_1) \cdot (1 + \delta_0 + \al) = e^{n_1 - n_0} (\al_1) \cdot ( \frac{\al_0}{e^{n_{1} - n_0}(\al_{1})} + \al) = \al_0 + e^{n_{1} - n_0}(\al_{1}) \cdot \al ,
\]
to conclude that
\[
\la n_0^{\delta_0 + \al} \ra \la n_1^{\delta_1} \ra \ldots \la n_k^{\al_k} \ra \top \equiv \la n_0^{\al_0 + e^{n_1-n_0} (\al_1) \cdot \al} \rangle \top \land \MNF{1}{k}.
\]
In particular, $\la n_0^\al \ra \ \Big(\la n_0^{\al_0} \ra \top \wedge \ldots \wedge \la n_k^{\al_k} \ra \top \Big) \ \vdash \ \la n_0^{\al_0 + e^{n_1-n_0} (\al_1) \cdot \al} \rangle \top $. \\
\medskip

For the other direction, note that 
\[
\begin{array}{ll}
\la n_0^{\al_0 + e^{n_1-n_0} (\al_1) \cdot \al} \rangle \top \ \land & \la n_0^\al \ra \, \Big( \la n_1^{\al_1} \ra \top \land \ldots \land \la n_k^{\al_k} \ra \top \Big) \\
\ & \ \ \ \ \vdash \ \ \la n_0^{\al_0 + e^{n_1-n_0} (\al_1) \cdot \al} \rangle \top \land \MNF{1}{k}.\\
\end{array}
\]
Thus, reusing reasoning from above we see that by Theorem \ref{MNFINC} together with Axiom \ref{coadditive}:
\[
\la n_0^{\al_0 + e^{n_1-n_0} (\al_1) \cdot \al} \rangle \top \land \MNF{1}{k} \equiv \la n_0^\al \ra \la n_0^{\delta_0} \ra \la n_1^{\delta_1} \rangle \ldots \la n_k^{\al_k} \ra \top.
\]
We now use Theorem \ref{INCMNF} and Rule \ref{r:3} again to obtain
\[
\la n_0^{\al_0 + e^{n_1-n_0} (\al_1) \cdot \al} \rangle \top \land \MNF{1}{k} \equiv \la n_0^\al \ra (\MNF{0}{k}).
\]
Thus,
\[
\begin{array}{ll}
\la n_0^{\al_0 + e^{n_1-n_0} (\al_1) \cdot \al} \rangle \top \ \land & \la n_0^\al \ra \, \Big( \la n_1^{\al_1} \ra \top \land \ldots \land \la n_k^{\al_k} \ra \top \Big) \\
\ &\vdash \ \  \la n_0^\al \ra \, \Big(\MNF{0}{k} \Big)\\
\end{array}
\]
which concludes the proof.
\end{proof}

Thus, this proposition allows us to drag a small modality over an {\sf MNF} that starts with the same modality. We note that the degenerate version of the proposition would be with $k=0$ which is covered by Axiom \ref{coadditive}:
\[
\la n_0^\alpha \ra \la n_0^{\alpha_0} \ra \top \ \ \equiv \ \ \la n_0^{\alpha_0 + \alpha} \ra \top.
\] 

\begin{theorem}\label{theorem:MNFsClosedUnderOrdinalModalities}
The set of {\sf MNF}s is, modulo provable equivalence, closed under putting ordinal modalities up front.
\end{theorem}

\begin{proof}
For an arbitrary {\sf MNF} $\MNF{0}{k}$ and arbitrary ordinal modality $\la n^\alpha \ra$ we see that $\la n^\alpha \ra \Big( \MNF{0}{k} \Big)$ is equivalent to some {\sf MNF}. The proof follows directly from our earlier results.

By Pseudo Schmerl 1, Proposition \ref{theorem:PseudoSchmerlOne}, we can `pull' the ordinal modality $\la n^\alpha \ra$ over the smaller ordinal modalities of $\MNF{0}{k}$ until $n\leq n_i$ for some $0\leq i \leq k$. 

In case $n=n_i$ we can apply Proposition \ref{theorem:smallModalityOverEqualMNF} (or the co-additivity axiom for the degenerate case) and we can apply directly Schmerl's axiom when $n<n_i$. The result yields us a conjunction of monomials which by Corollary \ref{theorem:conjunctionOfMonomialsHasMNF} we know is equivalent to some {\sf MNF}.
\end{proof}

\subsection{Each formula is equivalent to a normal form}

In this subsection, we shall combine all the earlier results of this sectio to obtain the, by now, simple corollary that every formula $\varphi$ is equivalent to a ${\sf MNF}$ formula. As a corollary we will obtain that $\varphi$ is also equivalent to an ordinal worm $A$ in ${\sf INF}$. However, to establish the uniqueness of ${\sf MNF}$'s and ${\sf INF}$'s the arithmetical soundness of the system is required. Hence, the uniqueness will be discussed later.

\begin{theorem} \label{MNFTHEOREM}
For every formula $\varphi$, there is $\psi \in {\sf MNF}$ such that $\varphi \equiv \psi$.
\end{theorem}
\begin{proof}
By induction on $\varphi$ applying Corollary \ref{theorem:MNFsClosedUnderConjunctions} and Theorem \ref{theorem:MNFsClosedUnderOrdinalModalities} for the conjunctive and ordianl modality cases, respectively.
\end{proof}

Consequently, we see that each formula also has an INF.

\begin{cor}
For every formula $\varphi$, there is $A \in {\sf INF}$ such that $\varphi \equiv A$.
\end{cor}

\subsection{Fine-tuning}
The current section contains some additional observations building forth on what we have obtained so far in this section. These results are not needed for the remainder of the paper.

We obtain a nice corollory from the proposition which we shall call \emph{Pseudo Schmerl 2}.

\begin{cor}\label{PS2}
Let $\MNF{0}{k}$ be in $\sf MNF$ with $k\geq 1$.
The following principle is derivable in $\TPr$:
\[
\begin{array}{llr}
\la n_0^\al \ra \, \Big(\, \la n_0^{\al_0} \ra \top & \wedge \ \ldots \ \wedge \la n_k^{\al_k} \ra \top \, \Big) & \text{(PS2)} \\
 & \equiv \la n_0^{\al_0 + e^{n_1-n_0} (\al_1) \cdot \al} \rangle \top \land \la n_1^{\al_1} \ra \top \land \ldots \land \la n_k^{\al_k} \ra \top.& \\
\end{array}
\]
\end{cor}

\begin{proof}
From Proposition \ref{theorem:smallModalityOverEqualMNF} we get that 
\[
\begin{array}{ll}
\la n_0^\al \ra \, \Big(\, \la n_0^{\al_0} \ra \top & \wedge \ \ldots \ \wedge \la n_k^{\al_k} \ra \top \, \Big)  \\
 & \equiv \la n_0^{\al_0 + e^{n_1-n_0} (\al_1) \cdot \al} \rangle \top \land 
\la n_0^\alpha \ra \, \Big(\,\la n_1^{\al_1} \ra \top \land \ldots \land \la n_k^{\al_k} \ra \top \, \Big). \\
\end{array}
\]
We conclude the proof by applying a Schmerl axiom, Axiom scheme \ref{MS1}, to 
\[
\la n_0^\alpha \ra \, \Big(\,\la n_1^{\al_1} \ra \top \land \ldots \land \la n_k^{\al_k} \ra \top \, \Big)
\]
observing that 
\[
\al_0 + e^{n_1-n_0} (\al_1) \cdot \al \geq e^{n_1-n_0} (\al_1) \cdot (1+ \al)
\]
since $\MNF{0}{k}$ is in $\sf MNF$, whence $\al_0 := e^{n_1-n_0} (\al_1) \cdot (2 + \beta)$ for some $\beta$.
\end{proof}

\begin{lemma}
Given $A := \la n^\al \ra \la m^\be \ra \top $, there is $\psi \in {\sf MNF}$ such that:
\begin{enumerate}
\item $A \equiv \psi$;
\item If $n < m$, then ${\sf N}$-${\sf mod}(\psi) = \{ n, m \}$ and  ${\sf O}$-${\sf mod}(\psi) = \{e^{m - n} (\be) \cdot (1 + \al), \be \}$;
\item If $n = m$, then ${\sf N}$-${\sf mod}(\psi) = \{ n \}$ and  ${\sf O}$-${\sf mod}(\psi) = \{ \be + \al \}$;
\item If $n > m$, then  ${\sf N}$-${\sf mod}(\psi) = \{ n \}$ and ${\sf O}$-${\sf mod}(\psi) = \{ \al \}$ or  ${\sf N}$-${\sf mod}(\psi) = \{ n, m \}$ and ${\sf O}$-${\sf mod}(\psi) \subseteq \{ \be + e^{n - m} (\al), \al\}$.
\end{enumerate}
\end{lemma}
\begin{proof}Proof goes by case distinction. The cases $n < m$ and $n = m$ are straightforward. For the case case $n > m$, notice that applying \hyperref[PS1]{(PS1)}, we have that 
\[
\la n^\al \ra \, \la m^\be \ra \top  \equiv \la m^{\be + e^{n - m} (\al)} \ra \top \, \land \, \la n^\al \ra \top. 
\] 
If $\be < e^{n - m} (\al)$, then 
\[
\la m^{\be + e^{n - m} (\al)} \ra \top \, \land  \, \la n^\al \ra \top \equiv \la n^\al \ra \top
\] 
which is in ${\sf MNF}$, and obviously ${\sf N}$-${\sf mod}(\psi) = \{ n \}$ and ${\sf O}$-${\sf mod}(\psi) = \{ \al \}$. Otherwise, notice that $\la m^{\be + e^{n - m} (\al)} \rangle \top \land \la n^\al \ra \top $ is indeed a ${\sf MNF}$ formula satisfying the required conditions.
\end{proof}

The following Proposition is obtained by means of Axiom \ref{MS1}, and principles \hyperref[PS1]{(PS1)} and \hyperref[PS2]{(PS2)}. The goal is to show how given an ordinal modality $\la n^\al \ra$ and $\psi \in {\sf MNF}$, we can obtain a new $\psi' \in {\sf MNF}$ that satisfies some conditions and which is equivalent to $\la n^\al \ra \, \psi$. 

\begin{proposition}
For any $\psi := \MNF{0}{k} \in {\sf MNF}$, with $k \geq 1$, and $0 < \al < \varepsilon_0$, there is $\psi' \in {\sf MNF}$ such that:
\begin{itemize}
\item $\la n^\al \ra \psi \equiv \psi' $;
\item If $n < n_0$, then: 
\begin{itemize} \item ${\sf N}$-${\sf mod}(\psi') = \{ n \} \cup {\sf N\text{-}mod}(\psi)$ and 
	  \item ${\sf O}$-${\sf mod}(\psi') = \{ e^{n_0 - n} (\al_0) \cdot (1 + \al)  \} \cup {\sf O\text{-}mod}(\psi)$;
\end{itemize}	  
\item If $n = n_0$,  then: 
\begin{itemize} \item ${\sf N}$-${\sf mod}(\psi') = {\sf N\text{-}mod}(\psi)$ and 
	  \item ${\sf O}$-${\sf mod}(\psi') = \big(\, {\sf O\text{-}mod}(\psi) - \{\al_0\} \, \big)  \cup \{ \al_0 + e^{n_1 - n_0} (\al_1 ) \cdot (\al) \}$;
	  \end{itemize}	
\item If $n > n_0$, then ${\sf N\text{-}mod}(\psi') \subseteq \{ n\} \cup {\sf N\text{-}mod}(\psi)$.
\end{itemize}	
\label{modNF}
\end{proposition}
\begin{proof}
Notice that for $n < n_0$ we just need to apply Axiom \ref{MS1}, and for $n = n_0$, we apply \hyperref[PS2]{(PS2)}. For $n > n_0$ we proceed by induction on $k$. For $k=1$, applying \hyperref[PS1]{(PS1)}, we get that:
\begin{tabbing}
$\la n^\al \ra \, \big(\, \la n_0^{\al_0}  \ra \top \, \land \, \la n_1^{\al_1} \ra \top \, \big)$ \= $\equiv \la n_0^{\al_0 + e^{n - n_0}(\al)} \ra \top \, \land \, \la n^\al \ra \la n_1^{\al_1} \ra \top$; \\[0.25cm]

\ \> $\equiv \la n_0^{\al_0 + e^{n - n_0}(\al)} \ra \top \wedge \psi$\\
\ \>  \hspace{1.5cm} with $\psi \in {\sf MNF}$, by previous Lemma.
\end{tabbing}
Consider the left-most monomial, $\la m^\gamma \ra \top$, occurring in $\psi$, that is, $$m := \min {\sf N}\text{-}{\sf mod}(\psi).$$ We distinguish the following cases:
\begin{enumerate}
\item If $\al_0 + e^{n - n_0}(\al) \leq e^{m - n_0}(\gamma)$ then  clearly, $\la n_0^{\al_0 + e^{n - n_0}(\al)} \ra \top \wedge \psi \equiv \psi$. 
\item If $\al_0 + e^{n - n_0}(\al) > e^{m - n_0}(\gamma)$, then consider the Cantor normal form of $\al_0 + e^{n_0 - n}(\al)$,  $[\, \al_0 + e^{n_0 - n}(\al) \, ]_{\sf CNF} := \delta_0 + \ldots + \delta_i$. We can distinguish two subcases: 
\begin{itemize} 
	  \item If $\delta_j \geq e^{m - n_0}(\gamma)$ for $j, \ 0 \leq j < i$, then $\la n_0^{\al_0 + e^{n_0 - n}(\al)} \ra \top \land \psi \in {\sf MNF}$ satisfying the conditions for ${\sf N\text{-}mod}$;
      \item If there is $j, \ 0 \leq j \leq i$ such that $\delta_j < e^{m - n_0}(\gamma)$ and for all $j' < j$, $\delta_{j'} \geq e^{m - n_0}(\gamma)$, then 
\[
\la n_0^{\al_0 + e^{n_0 - n}(\al)} \ra \top \, \land \, \psi \equiv \la n_0^{\delta_0 + \ldots + \delta_{j-1} + e^{m - n_0}(\gamma)} \ra \top \, \land \, \psi
\] 
where  $\la n_0^{\delta_0 + \ldots + \delta_{j-1} + e^{m - n_0}(\gamma)} \ra \top \in {\sf MNF}$ and satisfies the required conditions for ${\sf N\text{-}mod}$.
\end{itemize}
\end{enumerate}

For the inductive step, combining \hyperref[PS1]{(PS1)} and I.H., we obtain that:
\[
\la n^\al \ra \, \big(\, \MNF{0}{k+1} \, \big) \equiv \la n_0^{\al_0 + e^{n - n_0}(\al)} \ra \top \, \land \, \psi
\] 
where $\psi \in {\sf MNF}$. Then, we follow the same reasoning we used for the base case. 
\end{proof}

\section{Arithmetical soundness of \TPr} \label{arithint}

In this section we shall prove that the logic $\TPr$ adequately describes graded Turing progressions. First we shall make the link between our logic and graded Turing progressions more precise by introducing the \emph{Formalized Turing Progression (FTP) Interpretation}.

\subsection{The formalized Turing progression interpretation}

The Formalized Turing Progression Interpretation  (FTP) is the representation of provable $\TPr$ sequents as the entailment between the corresponding first order theories within $\EA^+$ via the $\Pi_2$-sentence expressing such derivability. Thus, given $\varphi, \, \psi \in \FLE$ the intended interpretation of $\varphi \vdash \psi$ is the arithmetical statement:
\[
\EA^+ \vdash \forall x \ (\, \Box_{\Th_\psi} (x) \rightarrow \Box_{\Th_\varphi} (x) \, ).
\]

Hereinafter, we will adopt the following notation: by $T \subseteq U$ we denote the arithmetical formula $\forall x \ (, \Box_{T} (x) \rightarrow \Box_{U} (x) \, )$. Likewise, by $T \equiv U$, we denote the formula $\forall x \ (\, \Box_{T} (x) \leftrightarrow \Box_{U} (x) \, )$ and $T \equiv_n U$ stands for the formula $\forall x \in \Pi_{n+1} \ (\, \Box_{T} (x) \leftrightarrow \Box_{U} (x) \, )$.\\

Before proving soundness of \TPr~under the FTP interpretation we will present some useful facts and tools in this subsection. 
In Lemma 2.2 of \cite{Beklemishev:1995:LocalRefVSCon} a very simple yet very useful fact is shown: if a formula is provable in some Turing progression, then it is provable in the base theory together with a single consistency statement.

\begin{lemma}  \label{tools}
Provably in $\EA^+$ we have that
\[
\forall \al \succ 0 \ \forall x \ \Big(\, \Box_{T_n^\al} (x) \rightarrow \exists \, \gamma {\prec} \al \ \Box_T \, \big(\,\con_n(\, (T)_n^{\dot{\gamma}} \,) \dot{\rightarrow} x \big) \,\Big).
\] 
\end{lemma}

The proof of this lemma is elementary as holds for the following.

\begin{lemma}\label{theorem:basicFactsOfTuringProgressions}
Let $T$ and $U$ be theories extending $\EA$. For any $n, m < \omega$ and $\delta \leq \al < \varepsilon_0$, provably in $\EA^+$:\label{BasicArith1}
\begin{enumerate}
\item \label{basic1} 
$T \subseteq U \ \to \  \big( \, \con_n (U)  \rightarrow \con_n (T) \,\big)$;

\item \label{p2}
$U \equiv_n T \ \ \rightarrow \ \ \big(\, \con_n(T) \leftrightarrow \con_n(U)\, \big)$; 

\item \label{noLabel}
$\beta \leq \gamma \ \rightarrow \ (T)^\beta_n \subseteq (T)^\gamma_n$;

\item \label{basic2}
$\con_n \big( \, (T)_m^\al \, \big) \rightarrow \con_n \big( \, (T)_m^\delta \, \big)$.
\end{enumerate} 
\end{lemma}

\begin{proof}
All items allow for elementary proofs. For example, reasoning informally in $\EA^+$, for Item \ref{basic1}, assume $T \subseteq U $ and $\con_n (U)$. Now, let $\pi \in \Pi_{n+1}$ such that $\Box_T \pi$. By assumption, $\Box_U \pi$ and by $\con_n (U)$ we get that ${\sf Tr}_{n+1}(\pi)$ as was to be shown. Item \ref{p2} goes analogously. Item \ref{noLabel} follows directly from Lemma \ref{tools} and Item \ref{basic2} follows from Item \ref{basic1} together with the fact that $(T)_m^\delta \subseteq (T)_m^\al$. 
\end{proof}

In a sense, it is a happy coincidence that we can prove these basic facts for Turing progressions. Various other basic facts would intuitively require transfinite induction. Of course, transfinite induction is in general not available in weak base theories. However, Schmerl (\cite{Schmerl:1978:FineStructure}) realized that in the context of provability logics, one can work with something quite similar but substantially weaker: \emph{reflexive induction} or sometimes called \emph{reflexive transfinite induction}.

\begin{proposition}[Reflexive induction]
For any p.r. well-ordering $(D,\prec)$, any theory $T$ containing $\EA^+$ is closed under the following reflexive induction rule:
\[
T \vdash \forall \al \ \Big(\, \Box_T \big(\forall \, \be {<} \dot{\al} \ \varphi(\be)\big) \ \ \rightarrow \ \ \varphi(\al)\Big) \ \ \Longrightarrow \ \ T \vdash \forall \al \, \varphi(\al).
\]
\end{proposition}

Reflexive induction is our main tool for proving facts about Turing progressions. It may seem very tricky but it has a very simple proof and is in a sense a direct consequence of L\"ob's theorem. Whenever we say that we are going to prove that $T$ proves $\forall \alpha \, \varphi (\alpha)$ by reflexive induction, we implicitly say that we are going to prove that in $T$ we can prove $\forall \al \ \Big(\, \Box_T \big(\forall \, \be {<} \dot{\al} \ \varphi(\be)\big) \, \rightarrow \, \varphi(\al)\Big)$. Further, we call the challange of proving $T\vdash \varphi (0)$ the base case, and for given $\alpha$, refer to the assumption in $T$ that $\Box_T \big(\forall \, \be {<} \dot{\al} \ \varphi(\be)\big)$ as the reflexive induction hypothesis (RIH).

With reflexive induction at hand, we have access to some form of transfinite induction even if we cannot prove well-foundedness of the corresponding orderings within the base theory. Let us, by way of example, proof with detail that stronger theories yield stronger Turing progressions.

\begin{lemma}\label{theorem:TuringProgressionsMonotoneInBaseTheory}
Let $T$ and $U$ be theories extending $\EA^+$, $n < \omega$ and $\al \prec \varepsilon_0$. Provably in $\EA^+$:
\[
T \subseteq U  \to (T)_n^\al \subseteq (U)_n^\al.
\]
\end{lemma}

\begin{proof}
We reason informally in $\EA^+$, assume $T\subseteq U$ and use reflexive transfinite induction on $\al$ with the base case being trivial. Thus, for $\alpha \succ 0$, consider an arbitrary formula $\chi$ such that $\Box_{(T)_n^\al} \chi$. Then, there is some $\delta \prec \alpha$ such that 
\[
\Box_{T} \big( \con_n \big(\, (T)_n^\delta \, \big) \to \chi \big).
\]
By the assumption that $T\subseteq U$ together with provable monotonicity of the provability operator, we get that $\Box_{(U)_n^\al} \big( \con_n ((T)_n^\delta) \to \chi \big)$. Moreover, we have that $\Box_{(U)_n^\al} \con_n ((U)_n^\delta)$. Now, by the RIH we know that $\Box_{\EA^+} \big(\,  (T)^\delta_n \subseteq (U)^\delta_n\,\big)$ so that along with Item \ref{basic1} of Lemma \ref{BasicArith1} under the $\Box_{\EA^+}$ we obtain $\Box_{(U)_n^\al} \con_n ((T)_n^\delta)$ and consequently $\Box_{(U)_n^\al} \chi$.
\end{proof}

From now on we shall include less details in our proofs that employ reflexive induction.

\begin{lemma}\label{theorem:smallGoesInside} 
Let $T$ be any theory extending $\EA^+$ and $\sigma \in \Sigma_{n+1}^0$. The following principles are provable in $\EA^+$: 
\begin{enumerate}
\item \label{t1}
$\con_{n} (\, T \,) \land \sigma \to \con_{n} (\, T + \sigma\, )$;

\item \label{t2}
$(T + \sigma)_{n}^\al \equiv (T)_{n}^\al + \sigma$;

\item \label{t3}
$T + \con_n \big(\,(T)_n^\al\, \big) \equiv (T)_n^{\al + 1}$;

\item \label{t4}
$\con_{n+1} (\, T \,) \ \to \ \con_n (\, T \,)$;

\item \label{t5}
$(T)_n^\al \subseteq (T)_{n+m}^\al$.
\end{enumerate}
\end{lemma}
\begin{proof}
Items \ref{t1}, \ref{t3} and \ref{t4} are easy to check. The right-to-left implication of Item \ref{t2} is straightforward. 

For the left-to-right implication we proceed by transfinite reflexive induction with the base case being trivial. Thus, we reason in $\EA^+$, assume as RIH that
\[
\Box_{\EA^+} \ \forall \be \prec \dot{\al} \ (\, (T + \sigma)_{n}^\be \subseteq (T)_{n}^\be + \sigma \, ),
\]
and let $\Box_{(T + \sigma)_{n}^\al} \chi $ for some arbitrary $\chi$. Thus, 
\[
\exists \delta \prec\alpha \ \Box_{T + \sigma} (\, \con_{n} ( ((T + \sigma)_{n}^\delta) ) \to \chi \, ),
\] 
and by monotonicity
\[
(*) \ \ \ \exists \delta \prec\alpha \ \Box_{(T)_{n}^\al + \sigma} (\, \con_{n} ( ((T + \sigma)_{n}^\delta) ) \to \chi \, ). \label{*t2}
\] 
Notice that $\Box_{(T)_{n}^\al + \sigma} (\, \con_{n} ( (T)_{n}^\delta ) \land \sigma \,)$ so with the help of Item \ref{t1} we obtain that $\Box_{(T)_{n}^\al + \sigma} \con_{n} ( (T)_{n}^\delta + \sigma)$. 

By the RIH together with monotonicity, we get that 
\[
\Box_{(T)_{n}^\al + \sigma} (\, (T + \sigma)_{n}^\delta \subseteq (T)_{n}^\delta + \sigma \, ).
\]
Combining this with Item \ref{basic1} of Lemma \ref{BasicArith1} (under a box) we conclude that  $\Box_{(T)_{n}^\al + \sigma} \con_{n} ( (T + \sigma)_{n}^\delta )$. With the help of \hyperref[*t2]{(*)}  we finally get that $\Box_{(T)_{n}^\al + \sigma} \ \chi$.

Item \ref{t5} tells us that Turing progressions are monotone in the consistency notion and follows from Item \ref{t4} from a straightforward reflexive transfinite induction.
\end{proof}

Our calculus will prove all the provable relations between the Turing progressions that are considered in this paper. A first step and corner stone in this study is Proposition \ref{Schmerl} below as was proven by Schmerl in \cite{Schmerl:1978:FineStructure} for the base theory of primitive recursive arithmetic. Beklemishev generalized this to the setting of $\EA^+$ with variable base theories. We cite Beklemishev's formulation here (Theorem 3 and Theorem 4 of \cite{Beklemishev:2003:AnalysisIteratedReflection}).

\begin{proposition}[Schmerl] \label{Schmerl}
Let $T$ be an axiomatizable theory which is a $\Pi_{n+1}$-axiomatized extension of $\EA$, then provably in $\EA^+$:
\begin{enumerate}[S1]
\item \label{S1}
$\ \ \ \forall \al \succ 0 \ \Big( \ (T)_{n+m}^\al \ \equiv_{n} \ (T)_n^{e^m (\al)} \ \Big)$; 

\item \label{S2}
$\ \ \ \forall \be \succ 0 \ \Big( \big( (T)_{n+m}^\be \big)_n^\al \ \equiv_{n} \ (T)_n^{e^m (\be) \cdot (1+\al)} \ \Big)$. 
\end{enumerate}
\end{proposition}

\subsection{Soundness}

In this subsection we shall show that \TPr is sound for the FTP interpretation. We first present the main simple argument and will then fill out the missing details in the remainder of this subsection.

\begin{theorem}[Soundness] \label{soundness}
Given $\varphi,\psi \in \FLE$, if $\varphi \vdash \psi$ then $$\EA^+ \vdash \forall x \ (\Box_{\Th_ \psi} (x) \rightarrow \Box_{\Th_\varphi} (x) \,).$$ 
\end{theorem}

\begin{proof}
By induction on the length of a $\TPr$ proof of $\varphi \vdash \psi$. Thus, we see that all axioms are arithmetically sound and that all rules preserve truth.

The first three axioms and first two rules are easily seen to be arithmetically sound. In particular, Axiom \ref{mon} --$\la n^\alpha \ra \varphi \vdash \la n^\beta \ra \varphi$ for $\alpha \geq \beta$-- expresses the property that smooth Turing progressions are monotone in the ordinal as expressed by Item \ref{noLabel} of Lemma \ref{theorem:basicFactsOfTuringProgressions}. On the other hand, Rule \ref{r:3} --if $\varphi\vdash {\psi}$, then $\la n^\al \ra \varphi\vdash \la n^\al \ra {\psi}$-- expresses the monotonicity of smooth Turing progressions in the base theory as reflected in Lemma \ref{theorem:TuringProgressionsMonotoneInBaseTheory}. For a proof of the soundness of the co-additivity axiom, Axiom \ref{coadditive}, we refer to Lemma 2.6 of \cite{Beklemishev:1995:LocalRefVSCon}. The remaining rule and the remaining two axioms are separately proven to be sound in the remainder of this subsection.
\end{proof}

We will start by looking at the remaining rule. Note that in proving the soundness we may use the arithmetical counterpart of the modal axioms and rules that we already have proven to be sound. Instead of talking every time about the arithmetical counterpart of such rules and axioms, we will for the sake of presentation simply speak about the modal rules and axioms. As such, our proofs below will seem an amalgamate of modal and arithmetical reasoning.

\begin{lemma}
Rule \ref{r:4} is sound w.r.t. the FTP interpretation. 
\end{lemma}

\begin{proof}
Let us recall Rule \ref{r:4}:
\[
\mbox{If $\varphi \vdash \psi$ then $\la n^\al \ra \varphi \ \land \ \la m^{\be + 1} \ra \psi \, \vdash \, \la n^\al \ra \, \Big( \, \varphi \ \land \ \la m^{\be + 1} \ra \psi \, \Big) \ \ \ \text{ for } m < n$. 
}
\]

To prove the arithmetical soundness, we assume that provably in $\EA^+$ we have $\T{\psi} \subseteq \T{\varphi}$. Consequenty $(\T{\psi} + \T{\varphi}) \subseteq \T{\varphi}$ so that by Rule \ref{r:3} we obtain 
\[
(*) \ \ \ (\T{\psi} + \T{\varphi} )_n^\al \subseteq (\T{\varphi})_n^\al.
\]
Also, we have the following provably in $\EA^+$:
\[
\Big(\, \T{\varphi} \, + \,(\T{\psi})_m^{\be + 1}\, \Big)_n^\al \ \subseteq \ \Big(\, \T{\varphi} \, + \, \T{\psi} \, + \ \con_m \big(\, (\T{\psi})_m^{\be} \, \big) \,\Big)_n^\al \ ,
\]
and since $\con_m\big( \, (\T{\psi})_m^{\be}\, \big) \in \Pi^0_{m+1}$ and $\Pi^0_{m+1} \subset \Sigma^0_{n+1}$ we also have
\[
\Big( \, \T{\varphi} \, + \, \T{\psi} \, + \ \con_m\big(\, (\T{\psi})_m^{\be}\, \big) \, \Big)_n^\al \ \subseteq \ \big(\, \T{\varphi} \, + \, \T{\psi} \, \big)_n^\al \, + \ \con_m\big( \, (\T{\psi})_m^{\be}\, \big).
\]

Consequently, for any formula $\chi$ we have provably in $\EA^+$ that 
\[
\Box_{\Big(\, \T{\varphi} \, + \,(\T{\psi})_m^{\be + 1}\, \Big)_n^\al}\, \chi \ \ \longrightarrow \ \ \Box_{(\T{\varphi} \, + \, \T{\psi} )_n^\al \, + \ \con_m( (\T{\psi})_m^{\be})} \, \chi \ .
\]

Now, by the formalized deduction theorem together with $(*)$ we have that for any formula $\chi$ that if
\[
\Box_{(\T{\varphi} \, + \, \T{\psi} )_n^\al \, + \ \con_m( (\T{\psi})_m^{\be})} \, \chi \ ,
\]
then $\Box_{(\T{\varphi} )_n^\al} \big( \con_m( (\T{\psi})_m^{\be}) \to \chi \big)$ whence by monotonicity
\[
\Box_{( \T{\varphi} )_n^\al \ + \ (\T{\psi})_m^{\be + 1}} \Big(\, \con_m \big(\, (\T{\psi})_m^{\be}\, \big) \to \chi \Big).
\]
In particular we get $\Box_{(\T{\varphi} )_n^\al \ + \ (\T{\psi})_m^{\be + 1}} \, \chi$ as was to be shown.
\end{proof}

We will now consider Axiom \ref{omega}:
$\la (m + n)^\al \ra \, \varphi \, \vdash \, \la m^{e^n (\al)} \ra \, \varphi$. It is easy to see that we can reduce the arithmetical soundness of this axiom to the arithmetical soundness of $\la (n+1)^\al \ra \varphi \vdash \la n^{e(\al)} \ra \varphi$. Thus, we just check the soundness of this principle with the help of the following proposition:

\begin{proposition} \label{BoxIndm}
Let $\varphi(x)$ be any arithmetical formula. By $I_\varphi (\bar{n})$ we denote the formula $\varphi(0) \ \land \ \forall x \ \big(\, \varphi(x) \to \varphi(s(x)) \, \big) \ \to \ \varphi(\bar{n})$. For any extension $T$ of $\EA$, provably in $\EA^+$:
\[
\forall n \ \Box_{T} I_\varphi (\bar{n}) .
\]  
\end{proposition}
\begin{proof}
Let $S_\varphi := \forall x \ \big(\, \varphi(x) \to \varphi(s(x)) \, \big)$, $S_{\varphi / n} := \varphi(\overline{n}) \to \varphi(\overline{n+1})$ and $P_\varphi := \varphi(0) \ \land \ S_\varphi$. Let $D(x,y)$ denote the recursive function such that given the G\"odel number of $\varphi$ is defined as follows: 
\begin{itemize}
 	\item $D(\Gn{\varphi}, 0) = \langle \, P_\varphi \to \ \varphi(0) \, \rangle$;
	\item \begin{tabbing}
			$D(\Gn{\varphi}, \overline{n+1}) = D(\Gn{\varphi}, \overline{n})^\frown \langle$ \= $S_\varphi \ \to \ S_{\varphi / n} $,\\ [0.25cm]
			\ \> $ (\, S_\varphi \, \to \, S_{\varphi / n} \, ) \ \to \ $ \\
			\ \>	 \hspace{1cm} $\big(\, P_\varphi \, \to \, (\, S_\varphi \, \to \, S_{\varphi / n} \, ) \, \big) $,\\[0.25cm]
			\ \> $P_\varphi \, \to \, (\, S_\varphi \, \to \, S_{\varphi / n} \, ) $, \\ [0.25cm]
			\ \> $P_\varphi \, \to \, S_\varphi$, \\ [0.25cm]
			\ \> $\big(P_\varphi \to ( S_\varphi \to  S_{\varphi / n}  )\big) \to $\\
			\ \> \hspace{1cm} $\big(\, (P_\varphi \to  S_\varphi) \to (P_\varphi \to S_{\varphi / n})\, \big)$, \\[0.25cm]
			\ \> $(P_\varphi \to  S_\varphi) \to (P_\varphi \to S_{\varphi / n})$, \\ [0.25cm]
			\ \> $P_\varphi \to S_{\varphi / n}$, \\[0.25cm]
			\ \> $P_\varphi \to S_{\varphi / n} \to$\\
			\ \> \hspace{1cm} $\big(\, (P_\varphi \to \varphi(\overline{n})) \to (P_\varphi \to \varphi(\overline{n+1})) \, \big)$, \\[0.25cm]
			\ \> $(P_\varphi \to \varphi(\overline{n})) \to (P_\varphi \to \varphi(\overline{n+1}))$, \\[0.25cm]
			\ \> $(P_\varphi \to \varphi(\overline{n+1}))\ \rangle$.
		   \end{tabbing}
	\end{itemize}
With this definition, we can easily check that provably in $\EA^+$.:
$$\forall n \ {\sf Prf}_T(\Gn{D(\varphi,\overline{n})},I_\varphi(\overline{n})) $$
and so,
$$\forall n \ \Box_{T} I_\varphi (\bar{n}) .$$
\end{proof}

\begin{lemma}
For any extension $T$ of $\EA^+$ any ordinal $\al \prec \varepsilon_0$ and natural number $n$, provably in $\EA^+$ we have:
\[
(T)_n^{e(\al)} \ \subseteq \ (T)_{n+1}^\al .
\]
\end{lemma}

\begin{proof}
By reflexive transfinite induction on $\al$ with the base case being trivial. Thus, reasoning in $\EA^+$, the RIH is
\[
\Box_{\EA^+} \ \forall \, \be {\prec} \dot{\al} \ \Big(\, (T )_{n}^{e(\be)} \ \subseteq \  (T)_{n+1}^\be \, \Big)
\]
as R.I.H. We proceed by a case distinction on $\al$. 

{\bf If $\al \in \suc$}, i.e., $\al := \be + 1$, let $\chi$ be some arbitrary formula such that $\Box_{(T)_n^{e(\be +1 )}} \, \chi$. Note that $e(\be + 1) = e(\be) \cdot \omega$. Therefore we have 
\[
\Box_{(T)_n^{e(\be +1 )}} \, \chi \ \longrightarrow \  \exists \, m {\prec} \omega \ \Box_T \Big(\, \con_n \big(\, (T)_n^{e(\be) \cdot m}\, \big) \to \chi \, \Big),
\]
and by monotonicity of 
provability we obtain
\[
\exists \, m {\prec} \omega \ \Box_{(T)_{n+1}^{\beta +1}} \Big(\, \con_n \big(\, (T)_n^{e(\be) \cdot m}\, \big) \to \chi \, \Big).
\]
Let $\theta (x) := \con_n \big(\, (T)_n^{e(\be) \cdot \dot x} \, \big)$. Clearly, we are done once we show that 
\begin{equation}\label{equation:thetaUnderBox}
\forall \, m {\prec} \omega \ \Box_{(T)_{n+1}^{\beta +1}} \, \theta(\dot m).
\end{equation} 
We will prove \eqref{equation:thetaUnderBox} by applying Proposition \ref{BoxIndm}. Therefore, it suffices to show $\Box_{(T)_{n+1}^{\be + 1}} \Big( \ \theta(0) \, \wedge \, \forall x \ ( \, \theta (x) \, \to \, \theta ( s(x) )   \, ) \Big)$. Clearly $\Box_{(T)_{n+1}^{\be + 1}} \, \con_n (T)$, that is, $\Box_{(T)_{n+1}^{\be + 1}} \, \theta(0)$. Thus, we only need to show $\Box_{(T)_{n+1}^{\be + 1}} \forall x \ \big( \, \theta (x) \, \to \, \theta ( s(x) )   \, \big) $.

In order to prove this, we reason in ${(T)_{n+1}^{\be + 1}}$,--which gets us under the $\Box_{(T)_{n+1}^{\be + 1}}$--suppose that $\con_n ( (T)_n^{e(\be)\cdot p})$ and set out to prove $\con_n ( (T)_n^{e(\be)\cdot {(p+1)}})$. 
Since we reason in ${(T)_{n+1}^{\be + 1}}$ and assume $\con_n ( (T)_n^{e(\be)\cdot p})$, we also have 
\[
\con_n \big( \, (T)_n^{e(\be)\cdot p}\, \big) \, \land \, \con_{n+1} \big( \, (T)_{n+1}^\be \, \big).
\] 
By Lemma \ref{theorem:smallGoesInside} we can push the small consistency inside the bigger ones and obtain
\begin{equation}\label{equation:PuntoEducacion}
\con_{n+1} \Big(\, \big(\,(T)_n^{e(\be)\cdot p}\, \big)_{n+1}^\be \Big).
\end{equation}
Since we are reasoning under a box, we can now apply the RIH without a box. However, we apply the RIH to the base theory $(T)_n^{e(\be)\cdot p}$ so that we obtain 
\begin{equation}\label{equation:PuntoEdu}
\big(\,(T)_n^{e(\be)\cdot p}\, \big)_n^{e(\beta)} \ \subseteq \  \big(\,(T)_n^{e(\be)\cdot p}\, \big)_{n+1}^\beta.
\end{equation}
We note that $\big(\,(T)_n^{e(\be)\cdot p}\, \big)_n^{e(\beta)}\  = \ (T)_n^{e(\be)\cdot p + {e(\beta)}} \ = \ (T)_n^{e(\be)\cdot {(p + 1)}}$. Via this equation, combining \eqref{equation:PuntoEdu} with \eqref{equation:PuntoEducacion} and realizing that consistency of a theory yields the consistency of any of its subtheories, we obtain $\con_{n+1} \big(\, (T)_n^{e(\be)\cdot {(p + 1)}} \, \big)$. By monotonicity we conclude $\con_{n} \big(\, (T)_n^{e(\be)\cdot {(p + 1)}} \, \big)$ so that inside $(T)_{n+1}^{\beta +1}$ we have shown $\con_{n} \big(\, (T)_n^{e(\be)\cdot {p}} \, \big) \ \to \ \con_{n} \big(\, (T)_n^{e(\be)\cdot {(p + 1)}} \, \big)$ as was to be shown.\\
\smallskip


{\bf{In the case that $\al \in {\text{lim}}$}} we assume $\Box_{(T)_n^{e(\al)}} \chi$. Since $e$ is a continuous function, as before we have that  
\[
\exists \, \be {\prec} \al \ \Box_{T} \Big(\, \con_n \big(\, (T)_n^{e(\be)} \, \big) \to \chi \, \Big).
\] 
By monotonicity we obtain for this $\beta$ that $\Box_{(T)^\alpha_{n+1}} \Big(\, \con_n \big(\, (T)_n^{e(\be)} \, \big) \to \chi \, \Big)$. Since $\Box_{(T)_{n+1}^\al} \con_n \big(\, (T)_{n+1}^{\be} \, \big)$ by the RIH we get  $\Box_{(T)_{n+1}^\al} \con_n \big(\, (T)_n^{e(\be)} \, \big)$ so that $\Box_{(T)_{n+1}^\al} \chi$.

\end{proof}

With respect to Axiom \ref{MS1}, the soundness proof rests on Item \ref{S2} of Proposition \ref{Schmerl} (Schmerl's formula).

\begin{lemma} \label{equivUnderSmall}
For any theories $T$ and $U$ extending $\EA$, $0 < m$ and $\al < \varepsilon_0$, if, provably in $\EA^+$, $T \equiv_{n+m} U$, then
\[
\EA^+ \vdash \con_n \big(\, (T)_n^\al \, \big) \leftrightarrow \con_n \big(\, (U)_n^\al \, \big).
\]
\end{lemma} 
\begin{proof}
Reasoning in $\EA^+$, suppose that $T \equiv_{n+m} U$. By reflexive transfinite induction on $\al$, assume as RIH $\Box_{\EA^+} \con_n \big(\, (T)_n^\be \, \big) \leftrightarrow \con_n \big(\, (U)_n^\be \, \big)$ for all $\be \prec \al$. \\

For the left-to-right implication, assume $\con_n \big(\, (T)_n^\al \, \big)$ and assume for $\pi \in \Pi^0_{n+1}$, $\Box_{(U)_n^\al} \pi$. Thus,
\[
\Box_{U} \, \Big( \, \con_n \big(\, (U)_n^\delta \, \big) \to \pi \, \Big)
\]
for $\delta \prec \al$. Since $\pi \in \Pi^0_{n+1}$ and $n < n + m$, by assumption,
\[
\Box_{T} \, \Big( \, \con_n \big(\, (U)_n^\delta \, \big) \to \pi \, \Big)
\]
and by monotonicity,
\[
\Box_{(T)_n^\al} \, \Big( \, \con_n \big(\, (U)_n^\delta \, \big) \to \pi \, \Big).
\]
Observe that
\[
\Box_{(T)_n^\al} \, \con_n \big(\, (T)_n^\delta \, \big),
\] 
and since by monotonicity the RIH holds provably in $(T)_n^\al$, we conclude that $\Box_{(T)_n^\al} \pi$. For the other other implication we reason analogously.

\end{proof}

The following propositions make use of the previous result in order to obtain conservativity results between {\sf INF}'s and {\sf MNF}'s.

\begin{proposition} \label{SoundSch1}
Given $\psi := \MNF{0}{k} \in {\sf MNF}$ and $A\in {\sf INF}$ such that $A = \mathcal{I}(\psi)$. We then have that provably in $\EA^+$:
\[
\Th_A \equiv_{n_0} (\EA^+)_{n_0}^{\al_0}.
\]
\end{proposition}

\begin{proof}
The proof goes by induction on $k$, with the base case being trivial. For the inductive step, recall that if $A = \mathcal{I}(A)$ then $A := \la n_0^{\be_0} \ra \ldots \la n_k^{\al_k} \ra \top$ where $\be_i = -1 + \frac{\al_i}{e^{n_{i+1} - n_i}(\al_{i+1})}$. Assume for some $\pi \in \Pi_{n_0 + 1}^0$ that $\Box_{\Th_A} \pi$. Hence, we have that
\[
\Box_{\big( \, \ldots (\EA^+)_{n_{k+1}}^{\al_{k+1}} \ldots \, \big)_{n_1}^{\be_1}} \Big( \, \con_{n_0} \big(\,( \ldots (\EA^+)_{n_{k+1}}^{\al_{k+1}} \ldots )_{n_0}^{\delta} \, \big) \ \, \to \ \, \pi \, \Big)
\]
for some $\delta \prec \be_0$. Notice that $\Big( \, \con_{n_0} \big(\,( \ldots (\EA^+)_{n_{k+1}}^{\al_{k+1}} \ldots )_{n_0}^{\delta} \, \big) \ \, \to \ \, \pi \, \Big) \in \Pi_{n_0 + 2}^0$ whence also in $\Pi_{n_1 + 1}^0$ so that by the external induction on $k$ we get that 
\[
\Box_{(\EA^+)_{n_1}^{\al_1}} \Big( \, \con_{n_0} \big(\,( \ldots (\EA^+)_{n_{k+1}}^{\al_{k+1}} \ldots )_{n_0}^{\delta} \, \big) \ \, \to \ \, \pi \, \Big).
\] 

By monotonicity it follows that  
\[
\Box_{((\EA^+)_{n_1}^{\al_1})_{n_0}^{\be_0}} \Big( \, \con_{n_0} \big(\,( \ldots (\EA^+)_{n_{k+1}}^{\al_{k+1}} \ldots )_{n_0}^{\delta} \, \big) \ \, \to \ \, \pi \, \Big).
\]
By the external IH, we have that $( \ldots (\EA^+)_{n_{k+1}}^{\al_{k+1}} \ldots )_{n_1}^{\be_1}\equiv_{n_1}(\EA^+)_{n_1}^{\al_1}$, and since $n_1 > n_0$, we can apply Lemma \ref{equivUnderSmall} obtaining that
\[
\con_{n_0}  \big( \, ( \ldots (\EA^+)_{n_{k+1}}^{\al_{k+1}} \ldots )_{n_0}^{\delta} \, \big) \leftrightarrow \con_{n_0} \big( \, ((\EA^+)_{n_1}^{\al_1})_{n_0}^{\delta} \, \big).
\]
Thus, since $\delta \prec \be_0$,  $\Box_{((\EA^+)_{n_1}^{\al_1})_{n_0}^{\be_0}} \pi$. Being $\pi \in \Pi_{n_0 + 1}^0$, we can apply Schmerl's formula obtaining 
\[
\Box_{(\EA^+)_{n_0}^{e^{n_1 - n_0}(\al_1) \cdot (1 + \be_0) }} \pi.
\] 
Finally, since $\psi \equiv \mathcal{I}(\psi)$ i.e. $\psi \equiv A$, we know that $e^{n_1 - n_0}(\al_1) \cdot (1 + \be_0) = \al_0$, therefore $\Box_{(\EA^+)_{n_0}^{\al_0 }} \pi$.\\

For the other direction, the reasoning is analogous. Assume $\Box_{(\EA^+)_{n_0}^{\al_0 }} \pi$ for some $\pi \in \Pi_{n_0 + 1}^0$. By Schmerl's formula, $\Box_{((\EA^+)_{n_1}^{\al_1})_{n_0}^{\be_0}} \pi$ and by the formalized deduction theorem:
\[
\Box_{(\EA^+)_{n_1}^{\al_1}} \Big( \, \con_{n_0} \big( \, ((\EA^+)_{n_1}^{\al_1})_{n_0}^{\delta} \, \big) \to \pi \, \Big)
\]
for some $\delta \prec \be_0$. Since $\Big( \, \con_{n_0} \big( \, ((\EA^+)_{n_1}^{\al_1})_{n_0}^{\delta} \, \big) \to \pi \, \Big)$ is in $\Pi_{n_1 + 1}^0$, by the external IH we obtain that 
\[
\Box_{\big( \, \ldots (\EA^+)_{n_{k+1}}^{\al_{k+1}} \ldots \, \big)_{n_1}^{\be_1}} \Big( \, \con_{n_0} \big( \, ((\EA^+)_{n_1}^{\al_1})_{n_0}^{\delta} \, \big) \to \pi \, \Big)
\]
and by monotonicity:
\[
\Box_{\big( \,( \, \ldots (\EA^+)_{n_{k+1}}^{\al_{k+1}} \ldots \, )_{n_1}^{\be_1} \, \big)_{n_0}^{\be_0}} \Big( \, \con_{n_0} \big( \, ((\EA^+)_{n_1}^{\al_1})_{n_0}^{\delta} \, \big) \to \pi \, \Big).
\]
As we reasoned before, by external IH and Lemma \ref{equivUnderSmall} we know that
\[
\con_{n_0}  \big( \, ( \ldots (\EA^+)_{n_{k+1}}^{\al_{k+1}} \ldots )_{n_0}^{\delta} \, \big) \leftrightarrow \con_{n_0} \big( \, ((\EA^+)_{n_1}^{\al_1})_{n_0}^{\delta} \, \big)
\]
Thus, $\Box_{\big( \,( \, \ldots (\EA^+)_{n_{k+1}}^{\al_{k+1}} \ldots \, )_{n_1}^{\be_1} \, \big)_{n_0}^{\be_0}} \pi$ i.e. $\Box_{\T{A}} \pi$.

\end{proof}

\begin{lemma}
Given $\psi := \MNF{0}{k} \in {\sf MNF}$ and $A\in {\sf INF}$ such that $A = \mathcal{I}(\psi)$, provably in $\EA^+$:

$$ \Th_A \equiv \Th_{\psi} .$$ \label{SoundSch2}
\end{lemma}
\begin{proof}
For the left-to-right implication, we formalized the following reasoning in $\EA^+$. Consider $\chi$ such that $\Box_ {\Th_A} \chi$. Therefore, by successively applying formalized deduction theorem, 
\[
\Box_{(\EA^+)_{n_k}^{\al_k}} \Big(\, \bigwedge_{0 \leq i < k} \con_{n_i} \big( \, (\ldots (T)_{n_k}^{\al_k} \ldots )_{n_i}^{\delta_i} \, \big) \to \chi \, \Big) 
\]
for some $\delta_i \prec \be_i$. And by monotonicity,
\[
\Box_{\Th_\psi} \Big(\, \bigwedge_{0 \leq i < k} \con_{n_i} \big( \, (\ldots (T)_{n_k}^{\al_k} \ldots )_{n_i}^{\delta_i} \, \big) \to \chi \, \Big) . 
\] 
With the help of Proposition \ref{SoundSch1}, we finally get that $\Box_{\Th_{\psi}} \chi.$
Right-to-left implication follows by an easy induction on $k$, applying Proposition \ref{SoundSch1}.
\end{proof}

With this previous result, we are now ready to prove the soundness of Axiom \ref{MS1}. 

\begin{lemma}Axiom \ref{MS1} is sound.
\end{lemma}
\begin{proof} Proof follows from Lemma \ref{SoundSch2} by means of the following reasoning within $\EA^+$. Let $\psi := \la n_0^{e^{n_0 - n} (\al_0) \cdot (1 + \al)} \ra \top \, \land \, \MNF{0}{k}$ and $\psi' := \MNF{0}{k} $. Since we have $\psi, \ \psi' \in {\sf MNF}$, let $A' = \mathcal{I}(\psi')$ and $A = \la n^\al \ra \, A' = \mathcal{I}(\psi)$. Thus, by Lemma \ref{SoundSch2} we have that
\[
\Th_{\psi} \equiv \Th_{A}
\]
and
\[
\Th_{\psi'} \equiv \Th_{A'}.
\]
Hence by necessitation,
\[
\Th_{A} \equiv \big( \, \Th_{\psi'} \, \big)_n^\al. 
\]
Therefore $\Th_{\psi} \equiv \big( \, \Th_{\psi'} \, \big)_n^\al$.

\end{proof}

\subsection{Conservativity, Uniqueness of {\sf MNF}'s and Completeness of $\TPr$} \label{boundCompl}

In this section, we shall prove the uniqueness of ${\sf MNF}$'s (Theorem \ref{MNFUNIQ}) together with some important facts that characterize the derivability in $\TPr$. With these results we shall show the completeness of our system in Theorem \ref{Compltnss}.\\

The following Corollary is a result of combining arithmetical soundness of $\TPr$ together with G\"odel's second incompleteness Theorem.

\begin{cor}
For any formula $\varphi \in \FLE$, $n < \omega$ and $\al < \varepsilon_0$, the sequent $\varphi \vdash \la n^\al \ra \varphi$ is not derivable in $\TPr$. \label{NonDerivable}
\end{cor}

\begin{proposition}
Let $\psi := \MNF{0}{k} \in {\sf MNF}$: \label{BoundsMNF} 
\begin{enumerate}
\item If $\psi \vdash \la m^\be \ra \top$ for some $\be > 0$, then $m \leq n_k$; \label{firstbound}
\item If $\psi \vdash \la n_i^\be \ra \top$ for $n_i \in {\sf N}\text{-}{\sf mod}(\psi)$, then $\be \leq \al_i$. \label{secondbound}
\end{enumerate}
\end{proposition}

\begin{proof}
For a proof of Item \ref{firstbound}, suppose that $\psi \vdash \la m^\be \ra \top$ for some $m > n_k$. Thus, we reason as follows:
\begin{tabbing}
$\la n_k^1 \ra \psi$ \= $\vdash \la n_k^1 \ra \la m^\be \ra \top$;\\[0.30cm]
\ \> $\vdash \Big( \, \la n_0^{\al_0 + e^{n_k - n0} (1)} \ra \top \ \land \ \ldots \ \land \ \la n_k^{\al_k + 1} \ra \top \, \Big) \ \land \ \la m^\be \ra \top$ by \hyperref[PS1]{(PS1)} ; \\[0.30cm]
\ \> $\vdash \la m^\be \ra \, \Big( \, \la n_0^{\al_0 + 1} \ra \top \ \land \ \ldots \ \land \ \la n_k^{\al_k + 1} \ra \top \, \Big)$; \\[0.30cm]
\ \> $\vdash \la n_0^{\al_0 + e^{m - n0} (\be)} \ra \top \ \land \ \ldots \ \land \ \la n_k^{\al_k + e^{m - n_k} (\be)} \ra \top$; \\[0.30cm]
\ \> $\vdash \la n_k^{e^{m - n_k} (\be)} \ra \psi$ by \hyperref[PS1]{(PS1)}.
\end{tabbing}
This yields a contradiction together with Axiom \ref{coadditive} and Corollary \ref{NonDerivable}. Analogously, for a proof of Item \ref{secondbound}, assume $\psi \vdash \la n_i^\be \ra \top$ for some $\be > \al_i$. For $i, \ 0 \leq i < k$, $\psi \vdash \la n_i^{\al_i + e^{n_{i+1}-n_i} (\al_{i+1})}\ra \top$, and so $\la n_i^1 \ra \psi \vdash \la n_i^{\al_i + e^{n_{i+1}-n_i} (\al_{i+1}) +1 }\ra \top$. We combine this with the following reasoning:
\[
\begin{array}{ll}
\la n_i^1 \ra \psi \\[0.20cm] 
& \hspace{-0.60cm} \vdash \Big( \, \la n_0^{\al_0 + e^{n_i - n_0} (1)} \ra \top \, \land \, \ldots \, \land \, \la n_i^{\al_i + e^{n_{i+1} - n_i} (\al_{i+1})} \ra \top \, \land \, \la n_{i+1}^{\al_{i+1}} \ra \top \, \land \, \ldots \\[0.20cm]
& \land \, \la n_k^{\al_k} \ra \top \, \Big) \, \land \, \la n_i^{\al_i + e^{n_{i+1}-n_i} (\al_{i+1}) +1 } \ra \top \ \text{by \hyperref[PS1]{(PS1)} and \hyperref[PS2]{(PS2)}; } \\[0.30cm]
& \hspace{-0.60cm} \vdash \Big(\, \la n_0^{\al_0 + e^{n_i - n_0} (1)} \ra \top \, \land \, \ldots \, \land \, \la n_i^{\al_i + e^{n_{i+1} - n_i} (\al_{i+1})} \ra \top \, \land \, \la n_{i+1}^{\al_{i+1}} \ra \top \, \land \, \ldots \\[0.20cm]
& \land \, \la n_k^{\al_k} \ra \top \, \Big) \ \land \la n_i^{\al_i + e^{n_{i+1}-n_i} (\al_{i+1}) \cdot 2 } \ra \top;\\[0.30cm]
&\hspace{-0.60cm} \vdash \la n_0^{\al_0 + e^{n_i - n_0} (2)} \ra \top \, \land \, \ldots \, \land \, \la n_i^{\al_i + e^{n_{i+1}-n_i} (\al_{i+1}) \cdot 2 } \ra \top \, \land \, \la n_{i+1}^{\al_{i+1}} \ra \top \ \land \, \ldots \\[0.20cm] 
&\land \, \la n_k^{\al_k} \ra \top \ \text{ by monotonicity}; \\ [0.30cm]
& \hspace{-0.60cm} \vdash \la n_i^2 \ra \psi \ \text{by \hyperref[PS1]{(PS1)} and \hyperref[PS2]{(PS2)}}.
\end{array}
\]
Which again, yields a contradiction. For $i = k$, we follow a similar strategy. First of all notice that 
\[
\la n_k^1 \ra \psi \vdash \la n_0^{\al_0 + e^{n_k - n_0} (1)} \ra \top \, \land \, \ldots \, \land \, \la n_k^{\be + 1} \ra \top,
\] 
and since $\be > \al_k$, $\be + 1  = \al_k + \gamma$ for some $\gamma \geq 2$. Thus, 
\[
\la n_k^1 \ra \psi \vdash \la n_0^{\al_0 + e^{n_k - n_0} (1)} \ra \top \, \land \, \ldots \, \land \, \la n_k^{\al_k + \gamma} \ra \top
\] 
and then we get that 
\[
\la n_k^1 \ra \psi \vdash \la n_0^{\al_0 + e^{n_k - n_0} (\gamma)} \ra \top \, \land \, \ldots \la n_k^{\al_k + \gamma} \ra \top.
\] 
Hence by \hyperref[PS1]{(PS1)}, $\la n_k^1 \ra \psi \vdash \la n_k^\gamma \ra \psi.$ 
\end{proof}

\begin{lemma}
For any $\psi := \MNF{0}{k} \in {\sf MNF}$: \label{monomialSeq}
\begin{enumerate}
\item If $\psi \vdash \la n_i^\delta \ra \top$ for $n_i \in {\sf N}\text{-}{\sf mod}(\psi)$, then $\la n_i^{\al_i} \ra \top \vdash \la n_i^\delta \ra \top$ \label{monomialSeqInMod};
\item If $\psi \vdash \la n^\gamma \ra \top$ for $n \not \in {\sf N}\text{-}{\sf mod}(\psi)$, then $\la n_j^{\al_j} \ra \top \vdash \la n^\gamma \ra \top$ and $\gamma \leq e^{n_j - n} (\al_j)$, where $n_j := \min\{  m \in {\sf N}\text{-}{\sf mod} : m \geq n \}$ \label{monomialSeqNotInMod}.
\end{enumerate}
\end{lemma}
\begin{proof}
For a proof of Item \ref{monomialSeqInMod}, suppose $\psi \vdash \la n_i^\delta \ra \top$ for some $n_i \in {\sf N}\text{-}{\sf mod}(\psi)$. Thus, by Proposition \ref{BoundsMNF} Item \ref{secondbound}, $\delta \leq \al_i$. Thus, $\la n_i^{\al_i} \ra \top \vdash \la n_i^\delta \ra \top$. For Item \ref{monomialSeqNotInMod}, it suffices to check  $\gamma \leq e^{n_j - n} (\al_j)$ for $n_j := \min\{  m \in {\sf N}\text{-}{\sf mod} : m \geq n \}$. Assume $\psi \vdash \la n^\gamma \ra \top$ for some $\gamma > e^{n_j - n} (\al_j) $. Then, $\psi \vdash \la n^{e^{n_j - n} (\al_j) \cdot 2} \ra \top$, so we can reason as follows:
\[
\begin{array}{ll}
\la n^1 \ra \psi \\[0.20cm] 
&\hspace{-0.60cm} \vdash \la n^{e^{n_j - n} (\al_j) \cdot 2 + 1} \ra \top; \\[0.30cm]

&\hspace{-0.60cm} \vdash \Big( \, \la n_0^{\al_0 + e^{n - n_0} (1)} \ra \top \, \land \, \ldots \, \land \, \la n^{e^{n_j - n} (\al_j) \cdot 2} \ra \top \, \land \, \la n_j^{\al_j} \ra \top \, \land \, \ldots  \, \\[0.20cm]
& \land \ \la n_k^{\al_k} \ra \top \, \Big) \, \land \, \la n^{e^{n_j - n} (\al_j) \cdot 2 + 1} \ra \top ; \\[0.30cm]

&\hspace{-0.60cm} \vdash \Big( \, \la n_0^{\al_0 + e^{n - n_0} (1)} \ra \top \, \land \, \ldots \, \land \, \la n^{e^{n_j - n} (\al_j) \cdot 2} \ra \top \, \land \, \la n_j^{\al_j} \ra \top \, \land \, \ldots  \\[0.20cm]
& \land \ \la n_k^{\al_k} \ra \top \, \Big) \, \land \, \la n^{e^{n_j - n} (\al_j) \cdot 3} \ra \top ; \\[0.30cm]

&\hspace{-0.60cm} \vdash \la n_0^{\al_0 + e^{n - n_0} (2)} \ra \top \, \land \, \ldots \, \land \, \la n^{e^{n_j - n} (\al_j) \cdot 3} \ra \top \, \land \, \la n_j^{\al_j} \ra \top \, \land \, \ldots \, \land \, \la n_k^{\al_k} \ra \top; \\[0.30cm]

&\hspace{-0.60cm} \vdash \la n^2 \ra \psi.
\end{array}
\]

This reason is analogous to the one used in proof of Proposition \ref{BoundsMNF} Item \ref{secondbound}, yielding a contradiction Corollary \ref{NonDerivable}.
\end{proof}

\begin{theorem}
\label{MNFUNIQ}
Given $\psi_1 , \psi_2 \in {\sf MNF}$, if $\psi_1 \equiv \psi_2$ then $\psi_1 = \psi_2$.
\end{theorem} 

\begin{proof}
Right-to-left direction is trivial. For the left-to-right direction, assume that $\psi_1 \equiv \psi_2$, where $\psi_1$ and $\psi_2$ are of the form  $\MNF{0}{k}$ and $\la m_0^{\be_0} \ra \top \, \land \, \ldots \, \land \, \la m_j^{\be_j} \ra \top$ respectively. W.l.o.g assume $j \leq k$. We show by induction on $j-i$ that $\la m_{j-i}^{\be_{j-i}} \ra \top = \la n_{k-i}^{\al_{k-i}} \ra \top $. For the base case, by a combination of Items \ref{firstbound} and \ref{secondbound} of Proposition \ref{BoundsMNF}, together with the assumption $\psi_1 \equiv \psi_2$, we have that $\la n_k^{\al_k} \ra \top = \la m_j^{\be_j} \ra \top$. Assume as IH that $\la m_{j-i}^{\be_{j-i}} \ra \top = \la n_{k-i}^{\al_{k-i}} \ra \top $. Again, if $m_{j -(i+1)} = n_{k - (i + 1)}$, by Proposition \ref{BoundsMNF} Item \ref{secondbound}, we are done. Thus, assume w.l.o.g that $m_{j -(i+1)} < n_{k - (i + 1)}$. Thus, by Lemma \ref{monomialSeq} Item \ref{monomialSeqNotInMod} we have that $\la m_{j -{i}}^{\be_{j-i}} \ra \top \vdash \la n_{k - (i + 1)}^{\al_{k - (i + 1)}} \ra \top$ \footnote{Note that since $m_{j - (i+1)} < n_{k - (i + 1)} < m_{j-i}$, and so $n_{k - (i + 1)} \not \in {\sf N}\text{-}{\sf mod}(\psi_2)$}, but by IH $\al_{k - (i + 1)} < e^{m_{k-i} - n_{k - (i+1)}} (\be_{j-i}) = e^{n_{k-i} - n_{k - (i+1)}} (\al_{k - i})$, contradicting the definition of ${\sf MNF}$. Hence, we have that for $i \leq j$, $\la m_{j-i}^{\be_{j-i}} \ra \top = \la n_{k-i}^{\al_{k-i}} \ra \top$. Moreover, if $j = k$, then clearly $\psi_1 = \psi_2$. Suppose $j < k$, since $\psi_1 \equiv \psi_2$, then $\psi_2 \vdash \psi_1$ and by Theorem \ref{MNFINC} we know that there is $A \in \WO$ such that $A \equiv \psi_2$ and $A \vdash \la n_0^\delta \rangle A$ for some $\delta > 0$, against Corollary \ref{NonDerivable}.
\end{proof}

\begin{lemma}
Let $\varphi \in \mathbb{F}_{< n}$ for some $n$, and $\psi \in {\sf MNF}$ such that $\varphi \equiv \psi$. Then $\psi \in \mathbb{F}_{<n}$.
\end{lemma}
\begin{proof}
Easy induction on $\psi$.
\end{proof}

\begin{proposition}
Given $\psi := \MNF{0}{k} \in {\sf MNF}$, provably in TSC: \label{Conservativity}
$$\MNF{0}{k} \equiv_{n} \bigwedge_{0 \leq i \leq j} \la n_i^{\al_i} \ra \top$$
where $n_j := \min\{ m \in {\sf N}\text{-}{\sf mod} : m \geq n \}$.
\end{proposition}
\begin{proof}
Left-to-right sequent is easy derivable. For right-to-left sequent, assume $\psi \vdash \chi$ for some $\chi \in \mathbb{F}_{<n+1}$. We prove our result by induction on $\chi$. Being base and conjunctive cases straightforward, we check the case $\chi := \la m^\be \ra \chi'$. Thus, there is $\psi' \in {\sf MNF}$ such that $\psi' \equiv \chi$. Moreover, by previous Lemma, since $\chi \in \mathbb{F}_{< n + 1}$, then $\psi' \in \mathbb{F}_{< n + 1}$. Let $\psi' := \la m_0^{\be_0} \ra \top \ \land \ \ldots \ \land \ \la m_l^{\be_l} \ra \top$ where $m_i \leq n \leq n_j$. By assumption $\psi \vdash \la m_0^{\be_0} \ra \top \ \land \ \ldots \ \land \ \la m_l^{\be_l} \ra \top$. Notice that for each $\la m_i^{\be_i} \ra \top$, either $m_i \in {\sf N}\text{-}{\sf mod}(\psi)$, and thus by Lemma \ref{monomialSeq} Item \ref{monomialSeqInMod}, $\la m_i^{\al_i} \ra \top$ occurs in $\psi$ and $\la n_{i'}^{\al_{i'}} \ra \top \vdash \la m_i^{\be_i} \ra \top$, or $m_i \not \in {\sf N}\text{-}{\sf mod}(\psi)$ and by Lemma \ref{monomialSeq} Item \ref{monomialSeqNotInMod} there is $\la n_{i'}^{\al_{i'}} \ra \top$ occurring in $\psi$ with $i' \leq j$, such that $\la n_{i'}^{\al_{i'}} \ra \top \vdash \la m_i^{\be_i} \ra \top$. Hence,
$$\bigwedge_{0 \leq i \leq j} \la n_i^{\al_i} \ra \top \vdash \psi'$$
that is,
$$\bigwedge_{0 \leq i \leq j} \la n_i^{\al_i} \ra \top \vdash \chi .$$
\end{proof}

\subsection{Arithmetical Completeness}

With the help of some of the results established in the previous subsection, we can state the following characterization of the derivability between formulas in ${\sf MNF}$. Furthermore, we can provide some derivable principles in $\TPr$, that together with the arithmetical soundness, will be useful in order to prove the arithmetical completeness.

\begin{lemma} \label{boundsImplyGodel}
Let $\psi := \MNF{0}{k} \in {\sf MNF}$. The following is derivable in $\TPr$:
\begin{enumerate}
\item If $\psi \vdash \la m^\be \ra \top$, $m > n_k$ and $\be > 0$, then $\la n_k^1 \ra \psi \vdash \la n_k^{e^{m - n_k} (\be)} \ra \psi$; \label{BIG1}
\item If $\psi \vdash \la m^\be \ra \top$, $\be > e^{n_i - m} (\al_i)$ and $n_i := \min \{ n \in {\sf N}\text{-}{\sf mod}(\psi_0) : n \geq m  \}$, then $\la n_i^1 \ra \psi \vdash \la n_i^2 \ra \psi$. \label{BIG2}
\end{enumerate}
\end{lemma}

\begin{proposition} \label{CharacterizatioMNFDeriv}
For any $\psi_0, \ \psi_1 \in {\sf MNF}$, where $\psi_0 := \MNF{0}{k}$ and $\psi_1 := \la m_0^{\be_0} \ra \top \ \land \ \ldots \ \land \ \la m_j^{\be_j} \ra \top$. $\psi_0 \vdash \psi_1$ iff the following holds:
\begin{enumerate}
\item $m_j \leq n_k$;
\item For any $m_l \in {\sf N}\text{-}{\sf mod}(\psi_1)$ and $n_i := \min \{ n \in {\sf N}\text{-}{\sf mod}(\psi_0) : n \geq m_l  \}$, $\be_l \leq e^{n_i - m_l} (\al_i)$.
\end{enumerate}

\end{proposition}

With these two results, we are ready now to prove completeness of $\TPr$:

\begin{theorem} \label{Compltnss}
For any $\varphi, \psi \in \FLE$, if $\EA^+ \vdash \  \Th_{\psi} \subseteq \Th_{\varphi} $, then the sequent $\varphi \vdash \psi$ is derivable in $\TPr$. 
\end{theorem}

\begin{proof} 
W.l.o.g., let $\varphi,\ \psi \in {\sf MNF}$, such that $\varphi := \MNF{0}{k}$ and $\psi := \la m_0^{\be_0} \ra \top \ \land \ \ldots \ \land \ \la m_j^{\be_j} \ra \top$. By contraposition, assume $\varphi \not \vdash \psi$. Hence by Theorem \ref{CharacterizatioMNFDeriv}, either:
\begin{itemize} 
\item $n_k < m_j$ or 
\item There is $m_l \in {\sf N}\text{-}{\sf mod}(\psi)$ and $n_i := \min \{ n \in {\sf N}\text{-}{\sf mod}(\varphi) : n \geq m_l  \}$, such that  $\be_l > e^{n_i - m_l} (\al_i)$.
\end{itemize}

Reasoning in $\EA^+$, towards a contradiction assume that $\Th_{\psi} \ \subseteq \ \Th_{\varphi}$. If $n_k < m_j$, observe that $(\EA^+)_{m_j}^{\be_j} \ \subseteq \ \Th_{\psi}$ and so, in particular  we have that $(\EA^+)_{m_j}^{\be_j} \ \subseteq \ \Th_{\varphi}$. Thus, applying Proposition \ref{boundsImplyGodel} Item \ref{BIG1} together with arithmetical soundness we obtain that $(\Th_{\varphi} )_{n_k}^{e^{m_j - n_k} (\be)}  \ \subseteq \ \big(\, \Th_{\varphi} + \con_{n_k} ( \Th_{\varphi} ) \, \big)$. Therefore,  
\[
\EA^+ \vdash \Box_{(\Th_{\varphi})_{n_k}^1} \con_{n_k} \big( \, (\Th_{\varphi})_{n_k}^1 \, \big).
\]
For the second case, we reason analogously but using Proposition \ref{boundsImplyGodel} Item \ref{BIG2}. 
\end{proof}

\section{Miscellanea}

In this section we collect some topics that we consider worth mentioning. We would like to include as a side note to this work.

\subsection{Modal Schmerl}

In this subsection we shall present a modal formulation of the Schmerl results introduced in Proposition \ref{Schmerl}. We give a modal proof without using the completeness result. A related point to consider is that under the arithmetical interpretation, this formulation and the Schmerl principles are not exactly the same.

\begin{proposition} For any $\varphi \in \mathbb{F}_{< n+2}$, the following principles are derivable $\TPr$:
\begin{enumerate}
\item $\la (n+1)^\al \ra \, \varphi \equiv_n \la n^{e(\al)} \ra \, \varphi$; \label{ModalSCH1}
\item $\la n^\al \ra \, \la (n+1)^\be \ra \, \varphi \equiv_n \la n^{e(\be) \cdot (1 + \al)} \ra \, \varphi$. \label{ModalSCH2}
\end{enumerate}
\end{proposition}

\begin{proof}
We only check Item \ref{ModalSCH2} since Item \ref{ModalSCH1} follows from Item \ref{ModalSCH2} by taking $\al = 0$. Let $\psi \in {\sf MNF}$ such that $\psi \equiv \varphi$. First we prove that the sequent 
\[
\la n^\al \ra \, \la (n+1)^\be \ra \, \varphi \vdash \la n^{e(\be) \cdot (1 + \al)} \ra \, \varphi
\] 
is derivable. By Theorem \ref{MNFTHEOREM} and monotonicity we have that:
\[
\begin{array}{ll}
\la n^\al \ra \, \la (n+1)^\be \ra \, \varphi \vdash &\Big( \, \bigwedge_{0 \leq i \leq k}  \, \la n_i^{\al_i + e^{n - n_i} ( e(\be) \cdot (1+\al)) } \ra \top  \, \Big) \ \land \\ [0.30cm]
& \hspace{0.60cm} \la n^{\gamma + e(\be) \cdot (1+\al)} \ra \top \ \land \ \la (n+1)^{\delta} \ra \top 
\end{array}
\]
where each $\la n_i^{\al_i} \ra \top$ occurs in $\psi$, and  $\gamma$ and $\delta$ might be $0$. Thus, by \hyperref[PS1]{(PS1)} we get that $\la n^\al \ra \, \la (n+1)^\be \ra \, \varphi \vdash \la n^{e(\be) \cdot (1 + \al)} \ra \, \varphi$.\\ 

For the other direction, consider $\chi \in \mathbb{F}_{n+1}$ such that $\la n^\al \ra \, \la (n+1)^\be \ra \, \varphi \vdash \chi$. Therefore, by \hyperref[PS1]{(PS1)} and \hyperref[PS2]{(PS2)} get the following:
\[
\begin{array}{ll}
\Big(\, \bigwedge_{0 \leq i \leq k} \la n_i^{\al_i + e^{n+1 - n_i} (\delta + \be) + e^{n - n_i} (\al) } \ra \top \, \Big) \ \land & \\ [0.30cm]
\hspace{1.20cm} \la n^{\gamma + e(\delta + \be) + e(\delta + \be) \cdot \al } \ra \top \ \land \  \la (n+1)^{\delta + \be} \ra \top&\vdash \chi
\end{array}
\] 
where again, each $\la n_i^{\al_i} \ra \top$ occurs in $\psi$, and  $\gamma$ and $\delta$ might be $0$. Thus,
\[
\la n_0^{\kappa_0} \ra \top \, \land \, \ldots \, \land \, \la n_k^{\kappa_k } \ra \top \, \land  \, \la n^{\gamma + e(\delta + \be) \cdot (1 +\al) } \ra \top \, \land \,  \la (n+1)^{\delta + \be} \ra \top \vdash \chi
\]
where $\kappa_k := \al_k + e^{n - n_k} \big(\, \gamma + e(\delta + \be) \cdot (1 +\al)\, \big)$ and $\kappa_i := \al_i + e^{n_{i+1} - n_i} (\kappa_{i+1})$. We can easily check that there is $\{ n_0, \ldots , n_j \} \subseteq {\sf N}\text{-}{\sf mod}\big(\, \la n_0^{\kappa_0} \ra \top \, \land \, \ldots \, \land \, \la n_k^{\kappa_k } \ra \top \, \big)$ \footnote{Observe that this subset may be empty.} such that:
\[
\la n_0^{\kappa_0} \ra \top \, \land \, \ldots \, \land \, \la n_j^{\kappa_j } \ra \top \, \land  \, \la n^{\gamma + e(\delta + \be) \cdot (1 +\al) } \ra \top \, \land \,  \la (n+1)^{\delta + \be} \ra \top \vdash \chi
\]
and $\la n_0^{\kappa_0} \ra \top \, \land \, \ldots \, \land \, \la n_j^{\kappa_j } \ra \top  \, \land  \, \la n^{\gamma + e(\delta + \be) \cdot (1 +\al) } \ra \top \, \land \,  \la (n+1)^{\delta + \be} \ra \top \in {\sf MNF}$. Let $\chi' \in {\sf MNF}$ such that $\chi' \equiv \chi$. Then, with the help of Proposition \ref{Conservativity} we get that:
\[
\la n_0^{\kappa_0} \ra \top \, \land \, \ldots \, \land \, \la n_j^{\kappa_j } \ra \top  \, \land  \, \la n^{\gamma + e(\delta + \be) \cdot (1 +\al) } \ra \top \vdash \chi'.
\]
Finally, notice that 
\[
\la n^{e(\be) \cdot (1 + \al)} \ra \, \psi \vdash \la n_0^{\kappa_0} \ra \top \, \land \, \ldots \, \land \, \la n_k^{\kappa_k } \ra \top \, \land  \, \la n^{\gamma + e(\delta + \be) \cdot (1 +\al) } \ra \top .
\]
This way, $\la n^{e(\be) \cdot (1 + \al)} \ra \, \psi \vdash \la n_0^{\kappa_0} \ra \top \, \land \, \ldots \, \land \, \la n_j^{\kappa_j } \ra \top  \, \land  \, \la n^{\gamma + e(\delta + \be) \cdot (1 +\al) } \ra \top $ and so 
\[
\la n^{e(\be) \cdot (1 + \al)} \ra \, \psi \vdash \chi'
\] 
i.e. $\la n^{e(\be) \cdot (1 + \al)} \ra \, \varphi \vdash \chi$.
\end{proof}

In the proof of the previous Proposition the assumption of $\varphi \in \mathbb{F}_{<n+2}$ is used to bound the greatest monomial occurring in the corresponding {\sf MNF} of $\varphi$. As we know, an homologous condition is present in the original arithmetical formulation. With the help of Corollary \ref{NonDerivable} it is easy to find some examples where the previous result does not hold when dropping this assumption. Consider the following counterexample: let $\varphi := \la 0^{\omega^\omega \cdot 2} \ra \top \, \land \, \la 2^1 \ra \top$. Thus, 
\[
\la 0^1 \ra \, \la 1^1 \ra \, \varphi \equiv \la 0^{\omega^{\omega \cdot 2} \cdot 2 } \ra \top \, \land \, \la 1^{\omega \cdot 2} \ra \top \, \land \, \la 2^1 \ra \top.
\] 
On the other hand, 
\[
\la 0^{\omega \cdot 2} \ra \, \varphi \equiv \la 0^{\omega^{\omega + 1} \cdot 2} \ra \top \, \land \, \la 2^1 \ra \top.
\] 
Therefore we have that $\la 0^1 \ra \, \la 1^1 \ra \, \varphi \vdash \la 0^{\omega^{\omega \cdot 2} \cdot 2 } \ra \top$ but $\la 0^{\omega \cdot 2} \ra \, \varphi \not \vdash \la 0^{\omega^{\omega \cdot 2} \cdot 2 } \ra \top $.

\subsection{Admissibility of Infinitary Rules}

In previous versions of this work, we gave a gave a presentation of $\TPr$ with the following infinitary rule:
$$\text{If for all } \be < \lambda \in \lim, \  \varphi \vdash \la n^\be \ra \, \psi \text{ then } \varphi \vdash \la n^\lambda \ra \, \psi .$$

Although the rule is arithmetically valid and allows us to give more simple version of some of the axioms, with the current presentation of $\TPr$ is not necessary the inclusion of this rule.

\section*{Acknowledgements}

The authors wish to thank Lev Beklemishev and Andr\'{e}s Cord\'{o}n Franco for their technical suggestions during the writting of this work.

\nocite{*}
\bibliographystyle{plain}
\bibliography{ref.bib}

\end{document}